\documentclass[12pt,twoside,english]{article}
\usepackage{amsfonts,babel,amssymb,amsmath,amsthm}
\usepackage{graphicx}

\numberwithin{equation}{section}

\newtheorem{proposition}{Proposition}[section]

\newtheorem{theorem}[proposition]{Theorem}
\newtheorem{hypothesis}{Hypothesis}

\newtheorem{remark}{Remark}[section]

\newcommand{\punto}{\,\,\cdot\,\,}

\newcommand{\smallfrac}[2]{{\textstyle\frac{#1}{#2}}}

\title{Strong coupling of finite element methods for the Stokes-Darcy problem\footnote{This research was partially supported 
by  Ministery of Education of Spain through the Project MTM2010-18427. }}

\author{
{\sc Antonio M\'arquez}\thanks{Departamento de Construcci\'on e Ingenier\'\i a de
Fabricaci\'on, Universidad de Oviedo, Oviedo, Espa\~na,
e-mail: {\tt amarquez@uniovi.es}}
$\,\,$
{\sc Salim Meddahi}\thanks{Departamento de Matem\'aticas, Facultad de Ciencias,
Universidad de Oviedo, Calvo Sotelo s/n, Oviedo, Espa\~na,
e-mail: {\tt salim@uniovi.es}}$\,\,$
{\sc Francisco-Javier Sayas}\thanks{Department of Mathematical Sciences,
University of Delaware, Newark DE 19716, USA, e-mail: {\tt
fjsayas@math.udel.edu}}
}

\date{}

\begin{document}

\maketitle

\begin{abstract}
The aim of this paper is to propose a systematic way to obtain 
convergent finite element schemes for the Darcy-Stokes flow problem by 
combining well-known mixed finite elements that are separately convergent for  
Darcy and Stokes problems. In the  approach in which the Darcy problem 
is set in its natural $\mathbf{H}(\text{div})$ formulation  
and  the Stokes problem is expressed in velocity-pressure form,  
the transmission condition ensuring global mass conservation becomes essential. 
As opposed to the strategy that  handles weakly this transmission condition 
through a Lagrange multiplier, we impose here this restriction  exactly 
in the space of global velocity field. 
Our analysis of the Galerkin discretization of the resulting problem 
reveals that, if the mixed finite element space used in the Darcy domain admits an
$\mathbf{H}(\text{div})$-stable discrete lifting of the normal trace, 
then it can be combined with any stable Stokes mixed finite 
element of the same order to deliver a stable global method with quasi-optimal convergence rate. 
Finally, we present a series of numerical tests confirming our theoretical convergence estimates.  
\end{abstract}

% \begin{keywords}
% mixed finite elements, Stokes problem, Darcy problem
% \end{keywords}

\section{Introduction}\label{sec:1}
In this paper we are interested by the mixed finite element approximation of the coupled 
Darcy-Stokes problem. The Darcy-Stokes coupled system provides a linear model 
for the simulation of incompressible flows in heterogeneous media. It   
is governed by the Stokes equations in one part of the domain, while in the other part, 
the flow is described by a standard second order elliptic equation derived from Darcy's law and 
conservation of mass. Proper transmission conditions must also be prescribed 
on the boundary common to the two media: conservation of mass enforces continuity of the normal 
velocities at this interface and conservation of momentum enforces the balance of 
the normal stresses. A further interface condition, supported by empirical evidence 
and known as the Beavers-Joseph-Saffman  condition, must also be taken into account, 
cf. \cite{beavers, saffman, jager}. 
%Strictly speaking, this third condition  
%is not a coupling condition in the sense that it does not relate involve  
%quantities from the two subdomains $\Omega_S$ and $\Omega_D$ , 
%but it is actually a kind of Robin 
%boundary condition on $\Sigma$ for the fluid problem. 

The development of suitable numerical methods for the Darcy-Stokes flow interaction  
has become a subject of increasing interest during the last decade. Discacciati et al.  
\cite{miglio} provided the first theoretical study of the problem. 
The Galerkin scheme discussed in \cite{miglio} is based on a standard 
finite element method for second order elliptic problems in the Darcy domain. 
Motivated by the need of obtaining direct finite element 
approximations of the Darcy flux, most of the approaches (cf. for instance \cite{layton, bernardi, 
GatMedOya09, karper, GatOyaSay11} ) consider now an $\mathbf{H}(\text{div})$-flux conforming formulation 
in the Darcy domain. Hence, in this case, a mixed variational formulation in the porous media is 
coupled with the usual velocity-pressure formulation in the Stokes domain. 
We also point out that other finite element discretization strategies have been considered 
for this problem. For instance,  a discontinuous Galerkin method is described in \cite{riviere} and a  
stabilized finite element method is presented in \cite{burman}.

In the mixed approach, the transmission condition that guaranties the equilibrium of the normal 
velocities becomes essential. One approach to deal with this constraint consists in enforcing it 
weakly by means of a Lagrange multiplier representing the trace of the Darcy pressure on the common 
interface, see for example \cite{layton, GatMedOya09, GatOyaSay11}. One known feature of the 
finite element discretization of this  formulation is that two meshes are required on the transmission 
boundary satisfying a stability condition between their corresponding mesh sizes, 
which certainly constitutes a quite cumbersome restriction.
  
In this paper we are interested in the strong coupling that considers the inclusion of the essential 
transmission condition  directly in the definition of the space to which 
the Darcy flux  and the fluid velocity belong (see \eqref{strong} below). This formulation prevents   
from introducing  a further unknown (the Lagrange multiplier)
simplifying by the way the saddle point structure of the problem. 
Two finite element discretizations have already been proposed for this formulation of the problem.
In the first one \cite{karper} the authors 
%use the finite element method introduced in \cite{mardal} to 
provide a unified approximation in the whole domain giving rise to a global 
$\mathbf{H}(\text{div})$-conforming 
scheme that is nonconforming in the fluid domain.
In the second one \cite{bernardi}, a mortar  discretization that combines the 
Bernardi-Raugel element and the lowest order Raviart-Thomas element is studied.

Our concern in this paper is to figure out how to combine any of the 
common mixed finite element methods for Darcy and Stokes problems to end up 
with a stable scheme for the coupled problem. We consider stable conforming 
Galerkin schemes in each subdomain and establish the rules under which these 
schemes can be matched to form a global convergent method. We will see that the only requirement 
for the viability of such a coupling is the existence of a discrete stable lifting of 
the normal trace in $\mathbf{H}(\text{div})$. We can prove that such a lifting is available
in the bi-dimensional case  for the Raviart-Thomas and the Brezzi-Douglas-Marini 
elements while, in the three-dimensional case, to obtain the same result, we need to assume that the triangulation 
is quasi-uniform in a neighborhood of the transmission boundary.  

The paper is organized as follows. In Section \ref{sec:2} we decribe  the governing equations, derive a mixed variational 
formulation of the problem and show that it is well-posed. Our Galerkin scheme is introduced in Section \ref{sec:3}  
by means of abstract finite elements spaces in each subdomain. We also establish in this section the conditions guarantying 
the stability of our approximation method. In Section \ref{sec:4}, we discuss the approximation properties of the 
discrete space for the global  velocity field and provide our main convergence result. The key hypotheses providing 
the convergence of our numerical scheme are discussed in Sections \ref{sec:5} and \ref{sec:6}. We show that they are satisfied 
for the most common  mixed finite elements  for Darcy and Stokes problems under mild conditions on the family of finite 
element triangulations. In Section \ref{sec:7} we provide some examples obtained by combining well-known finite elements for 
the Stokes problem, such as the MINI element and the Bernardi-Raugel element, with the Raviart-Thomas and the 
Brezzi-Douglas-Marini elements in order to obtain a global scheme for the coupled Darcy-Stokes problem.  
Finally, in Section \ref{sec:8}, we present  a set of numerical experiments that confirm
the converge rate predicted by the theory of the examples described in the former section.

\paragraph{Notation and background.}

Boldface fonts will be used to denote vectors  and vector valued functions. Also, if $H$ 
is a vector space of scalar functions, 
$\mathbf H$ will denote the space of $\mathbb R^d$ valued functions whose components are in $H$, 
endowed with the product norm. 

Given an integer $m\geq 1$ and a bounded Lipschitz domain $\mathcal O\subset \mathbb R^d$, $(d=2,3)$,  
we denote by $\|\cdot \|_{H^m(\mathcal O)}$ the norm in the usual Sobolev space 
$H^m(\mathcal O)$, cf. \cite{AdaFou03}. For economy of notation, $(\cdot,\cdot)_{\mathcal O}$ stands for the 
inner product in $L^2(\mathcal{O})$ and $\|\cdot \|_{\mathcal{O}}$ is the corresponding norm. 
We recall that $H^{1/2}(\partial \mathcal{O})$ represents the image of $H^1(\mathcal{O})$ by the trace operator. 
Its dual with respect to the pivot space $L^2(\partial \mathcal{O})$ is denoted $H^{-1/2}(\partial \mathcal{O})$. 
If $\Sigma$ is a part of the Lipschitz boundary $\partial \mathcal O$, 
we denote by $H_{00}^{1/2}(\Sigma)$ the space of functions from $H^{1/2}(\Sigma)$ 
whose extension by zero to the whole $\partial \mathcal{O}$ belongs to $H^{1/2}(\partial \mathcal{O})$. 
We will consider the dual $H_{00}^{-1/2}(\Sigma)$ of $H_{00}^{1/2}(\Sigma)$ with respect to the pivot space 
$L^2(\Sigma)$ and the angled bracket $\langle\cdot,\cdot\rangle_\Sigma$ will be used for the $L^2(\Sigma)$ inner 
product and its extension as the duality  product between $H_{00}^{-1/2}(\Sigma)$ and $H_{00}^{1/2}(\Sigma)$. 
For definition and basic properties of the space 
$\mathbf H(\mathrm{div},\mathcal O)$, we refer to \cite{GirRav86}. We will denote by 
$\mathbf H_{0}(\mathrm{div}, \mathcal O)$ the subspace of fields from $\mathbf H(\mathrm{div}, \mathcal O)$
with zero normal trace on the boundary $\partial \mathcal O$.

Given an open set $\mathcal O$ or the closure of an open set, $\mathbb P_k(\mathcal O)$ will denote 
the space of $d-$variate polynomials 
of degree not greater than $k$. The same notation will be applied for polynomials defined on flat 
$(d-1)$-dimensional open manifolds.
Finally, at the discrete level, the letter $h$ (with or without geometric meaning) will be used to 
denote discretization. The expression 
$a\lesssim b$ will be used to mean that there exists $C>0$ independent of $h$ such that $a\le C\,b$ for 
all $h$.

\section{Variational formulation}\label{sec:2}

The geometric layout of our problem is as follows: a domain $\Omega\subset \mathbb
R^d$ ($d=2$ or $d=3$) with polyhedral Lipschitz boundary is
subdivided into two subdomains by a Lipschitz polyhedral interface
$\Sigma$. The subdomains are denoted $\Omega_{\mathrm S}$ and
$\Omega_{\mathrm D}$ (S stands for Stokes and D for Darcy). We also
denote
\[
\Gamma_{\mathrm S}:=\partial\Omega_{\mathrm S}\setminus\Sigma \qquad
\Gamma_{\mathrm D}:=\partial\Omega_{\mathrm D}\setminus\Sigma.
\]
The normal vector field $\boldsymbol\nu$ on $\partial\Omega$ is
chosen to point outwards. We also denote by $\boldsymbol\nu$ the normal vector on $\Sigma$ 
that points from $\Omega_{\mathrm S}$ to $\Omega_{\mathrm D}$. We will add a
technical assumption on $\Sigma$ later on.

The strong form of the Stokes-Darcy system, as we will consider it here consists of the Stokes equations in $\Omega_{\mathrm S}$
\begin{eqnarray}\label{eq:0.a}
-2\nu\textbf{div}(\boldsymbol\varepsilon(\mathbf u_{\mathrm S}))+\nabla p_{\mathrm S}=\mathbf f_{\mathrm S} & & \mbox{in $\Omega_{\mathrm S}$},\\
\label{eq:0.b}
\mathrm{div}\,\mathbf u_{\mathrm S} = 0& & \mbox{in $\Omega_{\mathrm S}$},\\
\label{eq:0.c}
\mathbf u_{\mathrm S}=\mathbf 0 & & \mbox{on $\Gamma_{\mathrm S}$},
\end{eqnarray}
where $\boldsymbol\varepsilon(\mathbf u_{\mathrm S}):=\frac12(\boldsymbol\nabla \mathbf u_{\mathrm S}+(\boldsymbol\nabla \mathbf u_{\mathrm S})^\top)$, the Darcy equations in $\Omega_{\mathrm D}$,
\begin{eqnarray*}
\mathbf K^{-1}\mathbf u_{\mathrm D}+\nabla p_{\mathrm D}=\mathbf 0 & & \mbox{in $\Omega_{\mathrm D}$},\\
\mathrm{div}\,\mathbf u_{\mathrm D}= f_{\mathrm D}& & \mbox{in $\Omega_{\mathrm D}$},\\
\mathbf u_{\mathrm D}\cdot\boldsymbol\nu = 0 & & \mbox{on $\Gamma_{\mathrm D}$},\\
\int_{\Omega_{\mathrm D}}p_{\mathrm D}=0,
\end{eqnarray*}
and the coupling conditions
\begin{eqnarray}\label{eq:0.1}
\mathbf u_{\mathrm S}\cdot\boldsymbol\nu = \mathbf u_{\mathrm D}\cdot\boldsymbol\nu & & \mbox{on $\Sigma$},\\
\label{eq:0.2}
2\nu \boldsymbol\varepsilon(\mathbf u_{\mathrm S})\boldsymbol\nu - p_{\mathrm S}\boldsymbol\nu + \nu\kappa^{-1} \boldsymbol\pi_t 
\mathbf u_{\mathrm S} = -p_{\mathrm D}\boldsymbol\nu & & \mbox{on $\Sigma$},
\end{eqnarray}
where $\boldsymbol\pi_t \mathbf w:=\mathbf w-(\mathbf w\cdot\boldsymbol\nu)\boldsymbol\nu$.
The coupling conditions encode mass conservation, balance of normal forces and the Beavers-Joseph-Saffman condition. 
The matrix valued function $\mathbf K$ is assumed to be componentwise in $L^\infty(\Omega_{\mathrm D})$, symmetric 
and uniformly positive definite, so that $\mathbf K^{-1}$ is componentwise in $L^\infty(\Omega_{\mathrm D})$. 
The coefficient $\kappa\in L^\infty(\Sigma)$ is assumed to be  bounded from below by a positive constant 
a.e. on $\Sigma$. Because of the mass conservation 
condition across $\Sigma$, the homogeneous Dirichlet boundary condition for $\mathbf u_{\mathrm S}$ on $\Gamma_{\mathrm S}$ 
and the incompressibility condition in the Stokes domain, we can easily show that
\begin{equation}\label{eq:0}
\int_{\Omega_{\mathrm D}} f_{\mathrm D}=0
\end{equation}
is a necessary condition for existence of solution.

We set $\mathbf{H}^1(\Omega_{\mathrm{S}}):= H^1(\Omega_{\mathrm{S}})^d$ and consider the Sobolev spaces
\begin{eqnarray}\label{eq:1.10}
\mathbf H^1_{\mathrm S}(\Omega_{\mathrm S}) &:=& \{ \mathbf u \in
\mathbf{H}^1(\Omega_{\mathrm{S}}) \,:\, \mathbf u=\mathbf 0 \quad \mbox{on
$\Gamma_{\mathrm S}$}\}\\
\label{eq:1.11} \mathbf H_{\mathrm D}(\mathrm{div}, \Omega_{\mathrm
D}) &:=& \{ \mathbf u \in \mathbf H(\mathrm{div},\Omega_{\mathrm
D})\,:\, \mathbf u \cdot \boldsymbol\nu =0 \quad \mbox{on
$\Gamma_{\mathrm D}$}\},
\end{eqnarray}
endowed with their natural norms $\|\cdot\|_{\mathbf{H}^1(\Omega_{\mathrm{S}})}$ and 
$\|\cdot\|_{\mathbf H(\mathrm{div}, \Omega_{\mathrm D})}$ (cf. \cite{AdaFou03},
\cite{GirRav86}). We recall that the normal trace operator $\mathbf{u}\mapsto \mathbf{u}|_\Sigma \cdot \boldsymbol{\nu}$ 
is bounded from $\mathbf H_{\mathrm D}(\mathrm{div}, \Omega_{\mathrm D})$ onto $H_{00}^{-1/2}(\Sigma)$. 

We introduce the space for the velocity field
\begin{equation}\label{strong}
\mathbb X:=\{ \mathbf u=(\mathbf u_{\mathrm S},\mathbf u_{\mathrm
D})\in \mathbf H^1_{\mathrm S}(\Omega_{\mathrm S}) \times \mathbf
H_{\mathrm D}(\mathrm{div},\Omega_{\mathrm D})\,:\, \mathbf
u_{\mathrm S}\cdot\boldsymbol\nu=\mathbf u_{\mathrm
D}\cdot\boldsymbol\nu \quad \mbox{on $\Sigma$}\} \subset \mathbf
H_0(\mathrm{div},\Omega),
\end{equation}
and endow it with the product norm
\[
 \| \mathbf u \|_{\mathbb{X}} := \left( \| \mathbf u_{\mathrm S} \|_{\mathbf{H}^1(\Omega_{\mathrm{S}})}^2 +  
\| \mathbf u_{\mathrm D} \|_{\mathbf H(\mathrm{div}, \Omega_{\mathrm D})}^2 \right)^{1/2}, \quad 
\forall (\mathbf u_{\mathrm S},\mathbf u_{\mathrm D})\in \mathbb{X}. 
\]

The space for the pressure field is
\[
\mathbb Q:= L^2_0(\Omega_{\mathrm S}) \times L^2_0(\Omega_{\mathrm
D}) \times \mathbb R=:L^2_\star(\Omega) \times \mathbb R,
\]
where
\[
L^2_0(\mathcal O):= \{ p \in L^2(\mathcal O)\,:\, (p,1)_{\mathcal
O}=0\}.
\]
The pressure field is represented as $(p_{\mathrm S}, p_{\mathrm D},
\delta)$, where $p_{\mathrm S}+ \delta$ corresponds to the pressure
in the Stokes domain and normalization has been applied to have zero
integral of pressure over $\Omega_{\mathrm D}$. If we want to impose
zero integral of the pressure field over $\Omega$, we can easily
correct the result by adding a constant in a postprocessing step.
The space $\mathbb Q$ is endowed with the corresponding
product norm.

We next define two bounded bilinear forms for the weak formulation
of the Darcy-Stokes problem. We first consider $a: \mathbb X \times \mathbb X
\to \mathbb R$, given by
\[
a(\mathbf u,\mathbf v):= 2\nu(\boldsymbol\varepsilon(\mathbf
u_{\mathrm S}),\boldsymbol\varepsilon(\mathbf v_{\mathrm
S}))_{\Omega_{\mathrm S}}+\nu \langle\kappa^{-1}
\boldsymbol\pi_t\mathbf u_{\mathrm S},\boldsymbol\pi_t\mathbf
v_{\mathrm S}\rangle_\Sigma + (\mathbf K^{-1} \mathbf u_{\mathrm
D},\mathbf v_{\mathrm D})_{\Omega_{\mathrm D}}.
\]
We also consider $b:\mathbb X \times \mathbb Q \to \mathbb
R$ given by
\begin{equation}\label{eq:1.0a}
b(\mathbf u,(q,\delta)):= (\mathrm{div}\,\mathbf
u,q)_{\Omega_{\mathrm S}\cup \Omega_{\mathrm D}}+\delta \langle
\mathbf u_{\mathrm S}\cdot\boldsymbol\nu,1\rangle_\Sigma.
\end{equation}
For analytical purposes, this bilinear form can be considered as the
sum of two bilinear forms $b_1:\mathbb X \times L^2_\star(\Omega)
\to \mathbb R$ and $b_2:\mathbb X \times \mathbb R\to \mathbb R$:
\begin{equation}\label{eq:1.0b}
b_1(\mathbf u,q):=(\mathrm{div}\,\mathbf u,q)_{\Omega_{\mathrm S}\cup
\Omega_{\mathrm D}}, \qquad b_2(\mathbf
u,\delta):=\delta\langle\mathbf u_{\mathrm
S}\cdot\boldsymbol\nu,1\rangle_\Sigma.
\end{equation}
A variational form of the Darcy-Stokes problem 
looks for $(\mathbf u,(p,\delta)) \in \mathbb X \times \mathbb Q$
such that
\begin{equation}\label{eq:1.vp}
\begin{array}{rll}
a(\mathbf u,\mathbf v) - b(\mathbf v,(p,\delta)) &=(\mathbf
f_{\mathrm S},\mathbf v_{\mathrm S})_{\Omega_{\mathrm S}} & \forall
\mathbf v \in \mathbb X,\\[1.5ex]
b(\mathbf u,(q,\rho)) &=(f_{\mathrm D},q_{\mathrm
D})_{\Omega_{\mathrm D}} & \forall (q,\rho)\in \mathbb Q.
\end{array}
\end{equation}

\paragraph{Terminology.}
For economy of expression, given a bounded bilinear form $c:\mathbb
S\times \mathbb T \to \mathbb R$ (here $\mathbb S$ and $\mathbb T$
are Hilbert spaces), we will denote
\[
\mathrm{ker}\,c:=\{ s \in \mathbb S\,:\, c(s,t)=0 \quad \forall t
\in \mathbb T\}.
\]
(The set is actually the kernel of the operator $\mathbb S \ni
s\mapsto c(s,\cdot)\in \mathbb T^*$, which will remain unnamed.) We
will also say that $c$ satisfies the inf-sup condition for the pair
$\{ \mathbb S, \mathbb T\}$, whenever there exists $\beta>0$ such
that
\begin{equation}\label{eq:1.1}
\sup_{0\neq s \in \mathbb S}\frac{c(s,t)}{\|s\|_{\mathbb S}}\ge
\beta \| t\|_{\mathbb T}\qquad \forall t \in \mathbb T.
\end{equation}
This property is equivalent to surjectivity of the operator $s
\mapsto c(s,\punto)$ and implies the existence of a bounded
right-inverse of this operator. Actually, given any $\xi\in \mathbb
T^*$ we can find $s \in \mathbb S$ such that
\begin{equation}\label{eq:1.2}
c(s,t)=\xi(t) \quad \forall t \in \mathbb T \qquad \| s \|_{\mathbb
S}\le \beta^{-1} \|\xi\|_{\mathbb T^*}.
\end{equation}
The constant $\beta$ in \eqref{eq:1.1} and \eqref{eq:1.2} can be taken to be
the same and the map $\xi \mapsto s$ is linear.

\begin{proposition}\label{prop:1.1} For the bilinear forms given in
\eqref{eq:1.0a}--\eqref{eq:1.0b} it holds that
\begin{eqnarray*}
\mathrm{ker}\, b_1 &=& \{ \mathbf u \in \mathbb X\,:\,
\mathrm{div}\, \mathbf u \in \mathbb P_0(\Omega_{\mathrm S}) \times
\mathbb P_0(\Omega_{\mathrm D})\}\\
&=& \{ \mathbf u \in \mathbb X\,:\, \mathrm{div}\, \mathbf u \in
\mathrm{span}\{ |\Omega_{\mathrm S}|^{-1}\mathbf{1}_{\Omega_{\mathrm S}}-
|\Omega_{\mathrm D}|^{-1}\mathbf{1}_{\Omega_{\mathrm D}}\}\,\}\\
\mathrm{ker}\, b &=& \{ \mathbf u\in \mathbb X\,:\,
\mathrm{div}\, \mathbf u=0\},
\end{eqnarray*}
where $\mathbf{1}_{\Omega_{\mathrm D}}$ and $\mathbf{1}_{\Omega_{\mathrm S}}$ are 
characteristic functions of $\Omega_{\mathrm D}$ and $\Omega_{\mathrm S}$ respectively. 
\end{proposition}

\begin{proof}
Note that
\[
(\mathrm{div}\,\mathbf u_\circ,q_\circ)_{\Omega_\circ}=0\quad
\forall q\in L^2_0(\Omega_\circ) \qquad \Longleftrightarrow \qquad
\mathrm{div}\,\mathbf u_\circ \in \mathbb P_0(\Omega_\circ) \qquad
\circ \in \{\mathrm S,\mathrm D\}.
\]
Also, for an element $\mathbf u\in \mathbb X$
\[
(\mathrm{div}\,\mathbf u_{\mathrm S},1)_{\Omega_{\mathrm
S}}=\langle\mathbf u\cdot\boldsymbol\nu,1\rangle_\Sigma =
-(\mathrm{div}\,\mathbf u_{\mathrm D},1)_{\Omega_{\mathrm D}}.
\]
The proof of the statement is straightforward with help of these two
results.
\end{proof}

\begin{proposition}\label{prop:1.2} The bilinear form $b$ satisfies
the inf-sup condition for the pair $\{\mathbb X,\mathbb Q\}$.
\end{proposition}

\begin{proof}
Let
\[
\mathbb X_0:=\mathbf H^1_0(\Omega_{\mathrm S}) \times \mathbf
H_0(\mathrm{div},\Omega_{\mathrm D}) \subset \mathbb X.
\]
Note first that $b_1$ satisfies the inf-sup condition for the pair
$\{ \mathbb X_0,L^2_\star(\Omega)$\}, since this is equivalent to
separate and well-known inf-sup conditions in $\Omega_{\mathrm S}$
and $\Omega_{\mathrm D}$. Therefore $b_1$ satisfies the inf-sup
condition for the pair $\{ \mathbb X, L^2_\star(\Omega)\}$. The
inf-sup condition of $b_2$ for the pair $\{\mathbb X,\mathbb R\}$ is
equivalent to the existence of a vector field $\mathbf v \in \mathbb
X$ such that
\begin{equation}\label{eq:1.3}
\langle \mathbf v \cdot\boldsymbol\nu,1\rangle_\Sigma > 0.
\end{equation}
In fact, it is possible to construct $\mathbf v\in \mathbf
H^1_0(\Omega)$ with this property. The proof of the global inf-sup
condition for $b$ uses then the characterization given in
\cite[Theorem 2]{GatSay08}. We need to show that $\mathbb
X=\mathrm{ker}\, b_1+\mathrm{ker}\, b_2.$ Given $\mathbf v\in \mathbb X$
we can choose $\mathbf w\in \mathbb X_0$ such that
\begin{equation}\label{eq:1.4}
(\mathrm{div}\,\mathbf w,q)_\Omega=(\mathrm{div}\,\mathbf
v,q)_\Omega \qquad \forall q\in L^2_\star(\Omega).
\end{equation}
Note that $\mathbf w =\mathbf 0$ on $\Sigma$ and therefore $\mathbf
w \in \mathrm{ker}\, b_2$ and by definition $\mathbf v-\mathbf w \in
\mathrm{ker}\,b_1$, which proves the result.
\end{proof}

\begin{proposition}\label{prop:1.3}
Problem \eqref{eq:1.vp} is well posed.
\end{proposition}

\begin{proof}
Sufficient conditions for well-posedness of mixed problems are (see
\cite[Chapter 2]{BreFor91}): the inf-sup condition of $b$ for
$\{\mathbb X;\mathbb Q\}$ (shown in Proposition \ref{prop:1.2}) and
coercivity of $a$ in $\mathrm{ker}\, b$
\[
a(\mathbf u,\mathbf u)\ge \alpha\|\mathbf u\|_{\mathbb X}^2 \qquad
\forall \mathbf u \in \mathrm{ker}\, b.
\]
(In fact, since $a$ is symmetric and positive definite, these
conditions are also necessary.) However, the bilinear form $a$ is
coercive in the set
\begin{equation}\label{eq:1.5}
\mathbf H^1_{\mathrm S}(\Omega_{\mathrm S}) \times \{ \mathbf v \in
\mathbf H(\mathrm{div},\Omega_{\mathrm D})\,:\,
\mathrm{div}\,\mathbf v=0\},
\end{equation}
by Korn's inequality. By Proposition \ref{prop:1.1}, this set
includes $\mathrm{ker}\,b$, which proves the result.
\end{proof}

\section{The discrete problem}\label{sec:3}
We assume that there are two
separate families of regular triangulations $\{\mathcal T^h_{\mathrm S}\}_h$ and $\{\mathcal
T^h_{\mathrm D}\}_h$  of $\overline{\Omega}_{\mathrm S}$ and $\overline{\Omega}_{\mathrm D}$
respectively. Each triangulation is composed of
triangles/tetrahedra with a diameter not greater than $h$. The triangulations create two inherited
partitions of $\Sigma$, respectively denoted $\Sigma^h_{\mathrm S}$
and $\Sigma^h_{\mathrm D}$. 
Let us consider finite dimensional subspaces of piecewise polynomial
functions (relatively to the given triangulations) to approximate velocity
and pressure in the Stokes domain
\[
\mathbf H^h(\Omega_{\mathrm S}) \subset \mathbf H^1(\Omega_{\mathrm
S}), \qquad  L^h_0(\Omega_{\mathrm S})\subset L^2_0(\Omega_{\mathrm
S}),\qquad L^h(\Omega_{\mathrm S})=L^h_0(\Omega_{\mathrm S})\oplus
\mathbb P_0(\Omega_{\mathrm S}),
\]
as well as in the Darcy domain
\[
\mathbf H^h(\Omega_{\mathrm D}) \subset \mathbf
H(\mathrm{div},\Omega_{\mathrm D}), \qquad  L^h_0(\Omega_{\mathrm
D})\subset L^2_0(\Omega_{\mathrm D}),\qquad L^h(\Omega_{\mathrm
D})=L^h_0(\Omega_{\mathrm D})\oplus \mathbb P_0(\Omega_{\mathrm D}).
\]
The spaces
\[
\mathbf H^h_0(\Omega_{\mathrm S}):= \mathbf H^h(\Omega_{\mathrm
S})\cap \mathbf H^1_0(\Omega_{\mathrm S}), \qquad \mathbf
H^h_0(\Omega_{\mathrm D}):=\mathbf H^h(\Omega_{\mathrm D}) \cap
\mathbf H_0(\mathrm{div},\Omega_{\mathrm D})
\]
are the corresponding subspaces that are considered when applying
the discretization method to problems with homogeneous boundary
conditions on the entire boundary of each subdomain.
We also need to consider the spaces (recall \eqref{eq:1.10} and
\eqref{eq:1.11})
\begin{equation}\label{eq:3.4}
\mathbf H^h_{\mathrm S}(\Omega_{\mathrm S}):= \mathbf
H^h(\Omega_{\mathrm S}) \cap \mathbf H^1_{\mathrm S}(\Omega_{\mathrm
S}), \qquad \mathbf H^h_{\mathrm D}(\Omega_{\mathrm D}):= \mathbf
H^h(\Omega_{\mathrm D}) \cap \mathbf H_{\mathrm
D}(\mathrm{div},\Omega_{\mathrm D}),
\end{equation}
as well as the discrete spaces of normal components on $\Sigma$, namely, 
\[
\Phi^h_\circ:= \{ \mathbf u_h\cdot \boldsymbol\nu \,:\, \mathbf u_h
\in \mathbf H^h_\circ(\Omega_\circ)\}\subset L^2(\Sigma) \qquad
\circ\in \{\mathrm S,\mathrm D\}.
\]
In all what follows, we assume that $\Phi^h_{\mathrm D}$ contains at least the space of piecewise constant 
functions on $\Sigma^h_{\mathrm{D}}$, i.e.,
\begin{equation}\label{eq:6.1}
\mathbb P_0(\Sigma^h_{\mathrm{D}}):=\{ \phi_h :\Sigma \to \mathbb R\,:\, \phi_h|_e \in \mathbb P_0(e) 
\quad \forall e\in \Sigma^h_{\mathrm{D}}\}\subset \Phi^h_{\mathrm D}.
\end{equation}
We denote by $R^h_D$ the $L^2(\Sigma)$-orthogonal projection onto $\Phi^h_{\mathrm D}$.

In the forthcoming arguments, we will use the idea of a 
uniformly bounded right inverse of the normal trace $\mathbf H^h(\Omega_{\mathrm D})\ni \mathbf v_h 
\mapsto \mathbf v_h\cdot\boldsymbol\nu \in \Phi^h_{\mathrm D}$, i.e., a linear 
operator $\mathbf L^h: \Phi^h_{\mathrm D}\to \mathbf H^h_{\mathrm D}(\Omega_{\mathrm D})$ such that
\[
 \| \mathbf{L}^h \phi^h\|_{\mathbf{H}(\mathrm{div},\Omega_{\mathrm D})} \leq C \, \| \phi^h\|_{H_{00}^{-1/2}(\Sigma)}\quad \forall 
\phi^h\in \Phi^h_{\mathrm D}
\]
with a constant $C>0$ independent of $h$ and 
\begin{equation}\label{eq:3.6}
(\mathbf L^h\phi^h) \cdot \boldsymbol\nu=\phi^h\qquad
\forall\phi^h\in \Phi^h_{\mathrm D}.
\end{equation}
We will refer to such an operator as a stable lifting of the normal trace in $\mathbf H^h(\Omega_{\mathrm D})$.

The method we are proposing is a Galerkin discretization of the
variational problem \eqref{eq:1.vp} using the spaces
\begin{eqnarray*}
\mathbb X^h &:=&  \{ \mathbf u_h \equiv (\mathbf
u^h_{\mathrm S},\mathbf u^h_{\mathrm D}) \in \mathbf H^h_{\mathrm
S}(\Omega_{\mathrm S})\times \mathbf H^h_{\mathrm D}(\Omega_{\mathrm
D})\,:\,  \mathbf u^h_{\mathrm D}\cdot\boldsymbol\nu = R^h_D(\mathbf u^h_{\mathrm S}\cdot\boldsymbol\nu)\,\, \text{on $\Sigma$}\},\\
\mathbb Q^h &:=& L^h_0(\Omega_{\mathrm S})\times
L^h_0(\Omega_{\mathrm D}) \times \mathbb R=:L^h_\star(\Omega) \times
\mathbb R,
\end{eqnarray*}
that is, we look for $(\mathbf u_h,(p_h,\delta_h))\in \mathbb X^h
\times \mathbb Q^h$ such that
\begin{equation}\label{eq:2.0}
\begin{array}{rll}
a(\mathbf u_h,\mathbf v_h)-b(\mathbf v_h,(p_h,\delta_h)&=(\mathbf
f_{\mathrm S},\mathbf v_h)_{\Omega_{\mathrm S}} & \forall \mathbf
v_h \in \mathbb X^h,\\[1.5ex]
b(\mathbf u_h,(q_h,\rho_h))&=(f_{\mathrm D}, q_h)_{\Omega_{\mathrm
D}} & \forall (q_h,\rho_h) \in \mathbb Q^h.
\end{array}
\end{equation}

\begin{remark}
Note that  $\mathbb X^h \not\subset \mathbb X$ unless $\Phi^h_{\mathrm S}\subset \Phi^h_{\mathrm D}$ 
in which case \eqref{eq:2.0} becomes a conforming Galerkin approximation  of \eqref{eq:1.vp}.
\end{remark}
 
\paragraph{More terminology.} The discrete counterpart of the
kernel and inf-sup terminology introduced in Section \ref{sec:1} is
more or less straightforward to define. Let $\mathbb S_h \subset
\mathbb S$ and $\mathbb T_h \subset \mathbb T$ be sequences of
finite dimensional spaces of the Hilbert spaces $\mathbb S$ and
$\mathbb T$ and consider again a bounded bilinear form $c:\mathbb
S\times \mathbb T \to \mathbb R$. The discrete kernel of $c$ is the
set
\[
\mathrm{ker}\,c_h:=\{ s_h \in \mathbb S_h\,:\, c(s_h,t_h)=0 \quad
\forall t_h\in \mathbb T_h\},
\]
that is, it is the kernel of the discrete operator $\mathbb S_h\ni
s_h \mapsto c(s_h,\punto)\in \mathbb T_h^*$. We say that $c$
satisfies a {\bf uniform (discrete) inf-sup condition} for the pair
$\{ \mathbb S_h,\mathbb T_h\}$ when there exists $\beta >0$ such
that
\begin{equation}\label{eq:2.4}
\sup_{0\neq s_h \in \mathbb S_h}\frac{c(s_h,t_h)}{\| s_h\|_{\mathbb
S}}\ge \beta \|t_h\|_{\mathbb T} \qquad \forall t_h \in \mathbb T_h
\qquad \forall h.
\end{equation}
This condition implies that for all $\xi \in \mathbb T^*$ we can
find a sequence $s_h \in \mathbb S_h$ such that
\begin{equation}\label{eq:2.5}
c(s_h,t_h)=\xi(t_h) \quad \forall t_h \in \mathbb T_h,  \qquad \|
s_h\|_{\mathbb S}\le \beta^{-1}\|\xi\|_{\mathbb T^*}\qquad \forall h
\end{equation}
and the maps $\xi \mapsto s_h$ are linear.

Let us now discuss the properties of the discrete spaces
that ensure stability for \eqref{eq:2.0}. 
The first set of assumptions is natural for each of the sets of
discrete couples.
\begin{hypothesis}\label{hyp:1}
\begin{itemize}
\item[{\rm (a)}] The spaces for the discrete Stokes flow are stable
when applied to the Stokes problem with homogeneous Dirichlet 
boundary condition:
\begin{equation}\label{eq:2.1}
\sup_{\mathbf 0 \neq \mathbf u_h \in \mathbf H^h_0(\Omega_{\mathrm
S})} \frac{(\mathrm{div}\,\mathbf u_h, q_h)_{\Omega_{\mathrm
S}}}{\|\mathbf u_h\|_{\mathbf H^1(\Omega_{\mathrm S})}}\ge
\beta_{\mathrm S}\| q_h\|_{\Omega_{\mathrm S}}\qquad \forall q_h \in
L^h_0(\Omega_{\mathrm S}).
\end{equation}
\item[{\rm (b)}] The spaces for the discrete Darcy flow are stable when
applied to the Darcy equations with homogeneous normal trace:
\begin{equation}\label{eq:2.2}
\sup_{\mathbf 0 \neq \mathbf u_h \in \mathbf H^h_0(\Omega_{\mathrm
D})} \frac{(\mathrm{div}\,\mathbf u_h, q_h)_{\Omega_{\mathrm
D}}}{\|\mathbf u_h\|_{\mathbf H(\mathrm{div},\Omega_{\mathrm
D})}}\ge \beta_{\mathrm D}\| q_h\|_{\Omega_{\mathrm D}}\qquad
\forall q_h \in L^h_0(\Omega_{\mathrm D}).
\end{equation}
\item[{\rm (c)}] $\mathrm{div}\,\mathbf
H^h(\Omega_{\mathrm D}) \subset L^h(\Omega_{\mathrm D})$.
\end{itemize}
\end{hypothesis}

\begin{hypothesis}\label{hyp:2}
\begin{itemize}
\item[{\rm (a)}] There exist $\beta_0,\beta_1$ such that for all $h$
\begin{equation}\label{eq:2.3}
\exists \mathbf v_h \in \mathbf H^h_{\mathrm
S}(\Omega_{\mathrm S}) \quad \mbox{s.t.}\quad  
\langle\mathbf v_h\cdot\boldsymbol\nu,1\rangle_\Sigma\ge \beta_0 \quad \text{and}\quad 
 \|\mathbf v_h\|_{\mathbf H^1(\Omega_{\mathrm S})}\le \beta_1.
\end{equation}
\item[{\rm (b)}] There exists a stable lifting of the
normal trace.
%\item[{\rm (c)}] There exists a linear operator $R^h_D: L^2(\Sigma)
%\to \Phi^h_{\mathrm D}$ such that for all $\phi \in L^2(\Sigma)$
%\[
%\|R^h_D\phi\|_{L^2(\Sigma)}\lesssim\|\phi\|_{L^2(\Sigma)}\quad \forall
%\phi\in L^2(\Sigma)\qquad \langle R^h_D\phi-\phi,\xi_h\rangle_\Sigma=0
%\quad \forall \xi_h \in \Lambda^h .
%\]
%Equivalently
%\[
%\|\xi_h\|_{L^2(\Sigma)} \lesssim \sup_{0\neq \phi_h \in
%\Phi^h_{\mathrm
%D}}\frac{\langle\phi_h,\xi_h\rangle_\Sigma}{\|\phi_h\|_{L^2(\Sigma)}}\qquad
%\forall \xi^h \in \Lambda^h.
%\]
\end{itemize}
\end{hypothesis}

%\begin{hypothesis}\label{hyp:2}
%There exist $\beta_0, \beta_1>0$ such that for all $h$
%\begin{equation}\label{eq:2.3}
%\exists\mathbf v_h \in \mathbb X^h\qquad \mbox{s.t.}\qquad
%\langle\mathbf v_h\cdot\boldsymbol\nu,1\rangle_\Sigma \ge \beta_0
%\qquad \|\mathbf v_h\|_{\mathbb X}\le \beta_1.
%\end{equation}
%\end{hypothesis}
The validity of last set of hypotheses will be discussed in Sections \ref{sec:5} and \ref{sec:6}. 
Although at this stage these are just theoretical assumptions of the discrete spaces,
condition \eqref{eq:2.3} can be understood as the possibility of the
discrete spaces to create a certain amount of flux across $\Sigma$
with a bounded velocity field.

\begin{proposition}\label{prop:2.1} Hypotheses \ref{hyp:1} and \ref{hyp:2} imply that 
the bilinear form \eqref{eq:1.0a} satisfies a uniform inf-sup
condition for the pair $\{ \mathbb X^h,\mathbb Q^h\}$.
\end{proposition}

\begin{proof}
We will use the schematic method for proving inf-sup conditions in
product spaces given in \cite[Theorem 8]{GatSay08}. In order to do it, we
consider the decomposition $b=b_1+b_2$ given in \eqref{eq:1.0b}.
Let
\[
\mathbb X^h_0:=\mathbb X^h \cap \mathbb X_0= \mathbf
H^h_0(\Omega_{\mathrm S})\times \mathbf H^h_0(\Omega_{\mathrm D}).
\]
Conditions \eqref{eq:2.1} and \eqref{eq:2.2} in Hypothesis
\ref{hyp:1} are equivalent to $b_1$ satisfying a uniform inf-sup
condition for the pair $\{ \mathbb X^h_0, L^h_\star(\Omega)\}$. Therefore, they
imply the inf-sup condition of $b_1$ for $\{ \mathbb X^h,\mathbb L^h_\star(\Omega)\}$. 

On the other hand, we take $\mathbf v^h_{\mathrm S}$ as in Hypothesis \ref{hyp:2}(a) and
define $\mathbf v^h_{\mathrm D}:=\mathbf L^h\, R^h_D (\mathbf
v^h_{\mathrm S}\cdot\boldsymbol\nu)$, where $\mathbf L^h:
\Phi^h_{\mathrm D}\to \mathbf H^h_{\mathrm D}(\Omega_{\mathrm D})$
is the stable lifting of the normal trace. Then
\begin{eqnarray*}
\|\mathbf v^h_{\mathrm D}\|_{\mathbf H(\mathrm{div},\Omega_{\mathrm
D})} &\lesssim &  \|R^h_D(\mathbf v^h_{\mathrm
S}\cdot\boldsymbol\nu)\|_{H_{00}^{-1/2}(\Sigma)}\lesssim\|R^h_D(\mathbf
v^h_{\mathrm
S}\cdot\boldsymbol\nu)\|_\Sigma\\
&\lesssim &  \|\mathbf v^h_{\mathrm
S}\cdot\boldsymbol\nu\|_\Sigma\lesssim 1
\end{eqnarray*}
and $\mathbf v^h_{\mathrm D}\cdot\boldsymbol\nu =  R^h_D(\mathbf
v^h_{\mathrm S}\cdot\boldsymbol\nu)$ which implies that $(\mathbf v^h_{\mathrm S},\mathbf v^h_{\mathrm
D})\in \mathbb X^h$. We have then just shown that there exist $\beta_0, \beta'_1>0$ such that for all $h$
\begin{equation}\label{eq:5.2}
\exists \mathbf v_h\equiv (\mathbf v^h_{\mathrm S},\mathbf v^h_{\mathrm D})
\in \mathbb X^h \quad \mbox{s.t.}\quad
\langle\mathbf v^h_{\mathrm S}\cdot\boldsymbol\nu,1\rangle_\Sigma
\ge \beta_0 \quad \text{and}\quad  \|\mathbf v_h\|_{\mathbb X}\le
\beta'_1.
\end{equation}
It is easy to see that  \eqref{eq:5.2} is equivalent to
the uniform inf-sup condition of $b_2$ for $\{ \mathbb X^h,\mathbb
R\}$.

Let now $\mathbf v_h \in \mathbb X^h$. By the uniform inf-sup
condition of $b_1$ for $\{ \mathbb X^h_0,L^h_\star(\Omega)\}$ we can
find $\mathbf w_h \in \mathbb X^h_0$ such that
\[
(\mathrm{div}\,\mathbf w_h,q_h)_\Omega=(\mathrm{div}\,\mathbf
v_h,q_h)_\Omega \qquad \forall q_h \in L^h_\star(\Omega), \qquad
\|\mathbf w_h\|_{\mathbb X}\lesssim \|\mathrm{div}\,\mathbf
v_h\|_\Omega\le \|\mathbf v_h\|_{\mathbb X}.
\]
By definition, $\mathbf w_h \in \mathbb X^h_0\subset
\mathrm{ker}\,b_{2,h}$ and $\mathbf v_h -\mathbf w_h \in
\mathrm{ker}\,b_{1,h}$.
Therefore we can decompose $\mathbf v_h=\mathbf w_h+(\mathbf v_h-\mathbf w_h) \in \mathrm{ker}\,b_{2,h}
+ \mathrm{ker}\,b_{1,h} $, and this decomposition is stable. This is
enough to prove the uniform inf-sup condition of $b$.
\end{proof}

For the sake of measuring errors, we consider the space
\[
\widetilde{\mathbb X}:= \mathbf H^1_{\mathrm S}(\Omega_{\mathrm S})
\times \mathbf H_{\mathrm D}(\mathrm{div},\Omega_{\mathrm D}),
\]
that contains both $\mathbb X$ and $\mathbb X^h$. This
space is endowed with the same product norm as $\mathbb X$. In order to simplify the argument we will use a global
bounded bilinear form $A: (\widetilde{\mathbb X}\times \mathbb Q)
\times (\widetilde{\mathbb X}\times \mathbb Q) \to \mathbb R$ given
by
\[
A\big( (\mathbf u,(p,\delta),\, (\mathbf v,(q,\rho)\big) :=
a(\mathbf u,\mathbf v)-b(\mathbf v,(p,\delta))+b(\mathbf
u,(q,\rho)).
\]

\begin{proposition}[Stability and Strang estimate]\label{prop:5.1}
Hypotheses \ref{hyp:1} and \ref{hyp:2} imply unique solvability of
the discrete equations \eqref{eq:2.0} and the error estimate
\begin{eqnarray*}
& & \hspace{-3cm} \|\mathbf u-\mathbf u_h\|_{\mathbb
X}+\|p-p_h\|_{\Omega_{\mathrm
D}}+\|p+\delta-(p_h+\delta_h)\|_{\Omega_{\mathrm S}}\\
&\lesssim  &  \inf_{\mathbf v_h\in \mathbb X^h}\|
\mathbf u-\mathbf v_h\|_{\mathbb X}+\inf_{q_h \in
L^2_\star(\Omega)}\| p-q_h\|_\Omega + \mathrm{C}_h(p_{\mathrm
D}),
\end{eqnarray*}
where the consistency error term $C_h(p_{\mathrm D})\equiv 0$ if $\Phi_{\mathrm S}^h \subset \Phi_{\mathrm D}^h$ and 
\[
\mathrm C_h(p_{\mathrm D}) \lesssim h^{1/2} \|p_{\mathrm{D}} - R^h_D p_{\mathrm{D}} \|_\Sigma
\]
otherwise.
\end{proposition}

\begin{proof} We already know from Proposition \ref{prop:2.1} that the inf-sup 
condition of $b$ holds true for the pairs $\{ \mathbb X^h,\mathbb Q^h\}$. 
Hence, to prove that the operator
\[
\mathbb X^h\times \mathbb Q^h\ni \widehat
u^h\longmapsto A(\widehat u^h,\punto): \mathbb
X^h \times \mathbb Q^h\to \mathbb R
\]
has a uniformly bounded inverse we just have to show that the bilinear form
$a$ is coercive in the discrete kernel $\mathrm{ker}\, b_h:= \{ \mathbf u^h \in
\mathbb X^h\,:\,b(\mathbf u^h,(q_h,\rho_h))=0 \quad
\forall (q_h,\rho_h)\in \mathbb Q^h\}$.
By Hypothesis \ref{hyp:1}(c), if $\mathbf u_h \in
\mathrm{ker}\,b_h$, then $\mathrm{div}\,\mathbf u_h \in
L^h(\Omega_{\mathrm D})$ and since
\[
(\mathrm{div}\,\mathbf u_h,q_h)_{\Omega_{\mathrm D}}=0 \quad
\forall q_h \in L^h_0(\Omega_{\mathrm D}),
\]
it follows that $\mathrm{div}\,\mathbf u_h\in \mathbb
P_0(\Omega_{\mathrm D})$. Also, as part of the fact that $\mathbf
u_h \in \mathrm{ker}\,b_h$, it follows that
\[
(\mathrm{div}\,\mathbf u_{\mathrm{D}}^h,1)_{\Omega_{\mathrm D}} = 
\langle\mathbf u_{\mathrm{D}}^h\cdot\boldsymbol\nu,1\rangle_\Sigma=
\langle R^h_D(\mathbf u_{\mathrm{S}}^h\cdot\boldsymbol\nu),1\rangle_\Sigma=
\langle \mathbf u_{\mathrm{S}}^h\cdot\boldsymbol\nu, 1\rangle_\Sigma = 0
\]
where we used here the fact that
$\mathbb P_0(\Sigma) \subset \Phi_{\mathrm{D}}^h$.
Therefore, if $\mathbf u_h \in \mathrm{ker}\,b_h$, 
$\mathrm{div}\,\mathbf u_h =0$ in $\Omega_{\mathrm D}$. The coercivity  of $a$ on 
$\mathrm{ker}\, b_h$ follows then from the fact that it is coercive on
\[
\mathbf H^1_{\mathrm S}(\Omega_{\mathrm
S})\times \{ \mathbf v\in \mathbf H_{\mathrm
D}(\mathrm{div},\Omega_{\mathrm D})\,:\,\mathrm{div}\,\mathbf v=0\}.
\]

Let now $\widehat u:=(\mathbf u,(p,\delta))$ and $\widehat
u_h:=(\mathbf u_h,(p_h,\delta_h))$. Then, for all 
$\widehat v_h=(\mathbf v_h,(q_h,\rho_h))\in \mathbb
X^h\times\mathbb Q^h$,
\begin{equation}\label{eq:5.3}
A\big(\widehat u-\widehat u_h,\widehat v_h\big)=a(\mathbf u,\mathbf v_h)-b(\mathbf
v_h,(p,\delta))-(\mathbf f_{\mathrm S},\mathbf v_h)_{\Omega_{\mathrm
S}}
\end{equation}
and
\begin{equation}\label{eq:5.4}
a(\mathbf u,\mathbf v)-b(\mathbf v,(p,\delta))-(\mathbf f_{\mathrm
S},\mathbf v_{\mathrm S})_{\Omega_{\mathrm S}}=-\langle \mathbf
v_{\mathrm S}\cdot\boldsymbol\nu-\mathbf v_{\mathrm
D}\cdot\boldsymbol\nu,p_{\mathrm D}\rangle_\Sigma \qquad \forall
\mathbf v\in \widetilde{\mathbb X}.
\end{equation}
In this last formula, we have used that if $(\mathbf
u,(p,\delta))\in \mathbb X\times \mathbb Q$ is the solution of the
equations \eqref{eq:1.vp}, then $p_{\mathrm D} \in
H^1(\Omega_{\mathrm D})$. Therefore, by \eqref{eq:5.3} and  \eqref{eq:5.4},
\begin{equation}\label{eq:5.5}
A(\widehat u-\widehat u_h,\widehat v_h)=-\langle (\mathbf v_{\mathrm S}^h-\mathbf
v_{\mathrm D}^h)\cdot\boldsymbol\nu,p_{\mathrm D}\rangle_\Sigma.
\end{equation}

The discrete inf-sup condition satisfied by the global bilinear form $A$ and \eqref{eq:5.5} 
show that for all $\widehat w_h\in \mathbb X^h\times\mathbb Q^h$,
\begin{eqnarray*}
\| \widehat u-\widehat u_h\|_{\mathbb X\times \mathbb
Q}&\le & \|\widehat u-\widehat w_h\|_{\mathbb X\times
\mathbb Q}+\|\widehat u_h-\widehat w_h\|_{\mathbb X \times \mathbb Q}\\
& \lesssim & \|\widehat u-\widehat w_h\|_{\mathbb X\times
\mathbb Q}+ \sup_{0\neq \widehat v_h \in \mathbb X^h\times\mathbb Q^h} 
\frac{A(\widehat u_h-\widehat w_h,\widehat
v_h)}{\| \widehat v_h\|_{\mathbb X \times \mathbb Q}}\\
& \lesssim &  \| \widehat u-\widehat w_h\|_{\mathbb X \times
\mathbb Q}+ \sup_{0 \neq \widehat v_h \in \mathbb X^h\times\mathbb Q^h} 
\frac{A(\widehat u-\widehat u_h,\widehat v_h)}{\| \widehat
v_h\|_{\mathbb X \times \mathbb Q}}\\
& \lesssim &  \| \widehat u-\widehat w_h\|_{\mathbb X \times
\mathbb Q}+ \sup_{0 \neq \widehat v_h \in \mathbb X^h\times\mathbb Q^h} 
\frac{ | \langle (\mathbf v^h_{\mathrm S}-\mathbf v^h_{\mathrm D})\cdot\boldsymbol\nu, p_{\mathrm D}\rangle_\Sigma | }
{\| \widehat v_h\|_{\mathbb X\times \mathbb Q}}.
\end{eqnarray*}

It is clear that, if $\Phi_{\mathrm S}^h \subset \Phi_{\mathrm D}^h$, the last term of the previous inequality vanishes identically. 
In the general case we have to estimate the consistency error
\[
 C_h(p_{\mathrm{D}}):= \sup_{0 \neq \widehat v_h \in \mathbb X^h\times\mathbb Q^h} 
\frac{ | \langle (\mathbf v^h_{\mathrm S}-\mathbf v^h_{\mathrm D})\cdot\boldsymbol\nu, p_{\mathrm D}\rangle_\Sigma | }
{\| \widehat v_h\|_{\mathbb X\times \mathbb Q}}.
\]
To this end we introduce the $L^2(\Sigma)$-projection $\Pi^h_0:L^2(\Sigma)\to \mathbb P_0(\Sigma^h_{\mathrm{D}})$ 
onto the space of piecewise constant functions and denote $\boldsymbol\Pi^h_0:L^2(\Sigma)^d\to \mathbb P_0(\Sigma^h_{\mathrm{D}})^d$ 
its vectorial counterpart.  It is straightforward that
\begin{eqnarray*}
\langle (\mathbf v^h_{\mathrm S}-\mathbf v^h_{\mathrm D})\cdot\boldsymbol\nu, p_{\mathrm D}\rangle_\Sigma&=&
\langle (\mathbf v^h_{\mathrm S}-\mathbf v^h_{\mathrm D})\cdot\boldsymbol\nu, p_{\mathrm D}-R^h_D p_{\mathrm D}\rangle_\Sigma
= \langle \mathbf v^h_{\mathrm S}\cdot\boldsymbol\nu, p_{\mathrm D}-R^h_D p_{\mathrm D}\rangle_\Sigma\\
&=& \langle \mathbf v^h_{\mathrm S}\cdot\boldsymbol\nu-\Pi^h_0(\mathbf v^h_{\mathrm S}\cdot\boldsymbol\nu), 
p_{\mathrm D}-R^h_D p_{\mathrm D}\rangle_\Sigma\\
&=& \langle( \mathbf v^h_{\mathrm S}-\boldsymbol\Pi^h_0\mathbf v^h_{\mathrm S})\cdot\boldsymbol\nu, 
p_{\mathrm D}-R^h_D p_{\mathrm D}\rangle_\Sigma\\
&\lesssim & \|  \mathbf v^h_{\mathrm S}-\boldsymbol\Pi^h_0\mathbf v^h_{\mathrm S}\|_\Sigma\|
 p_{\mathrm D}-R^h_D p_{\mathrm D}\|_\Sigma\\
&\lesssim& h^{1/2} \|\mathbf v^h_{\mathrm S}\|_{\mathbf{H}_{00}^{1/2}(\Sigma)}\|
 p_{\mathrm D}-R^h_D p_{\mathrm D}\|_\Sigma,
\end{eqnarray*}
where in the last inequality we have applied a well known approximation estimate for piecewise constant functions. 
Using the trace theorem in $\mathbf{H}^1(\Omega_{\mathrm{S}})$ we deduce that the consistency error may be bounded by
\[
\mathrm C_h(p_{\mathrm D})\lesssim  h^{1/2}\|
 p_{\mathrm D}-R^h_D p_{\mathrm D}\|_\Sigma
\]
and the result follows.
\end{proof}

\section{Approximation properties of $\mathbb X^h$}\label{sec:4}
In principle, the restriction of equal normal component on $\Sigma$ can
reduce the size of the separate spaces ($\mathbf H^h(\Omega_{\mathrm
S})$ and $\mathbf H^h(\Omega_{\mathrm D})$) and limit the
approximation properties, unless appropriate matching on the
boundary allows for a rich enough discrete space. Note that (as
proven in Proposition \ref{prop:5.1}), this is not a matter of
stability, but of approximation.

\begin{proposition}\label{prop:5.4}
Hypothesis \ref{hyp:2}(b) implies that for all $\mathbf u \in \mathbb
X$
\begin{multline}\label{eq:mia}
\inf_{\mathbf v_h \in \mathbb X^h} \| \mathbf
u-\mathbf v_h\|_{\mathbb X} \lesssim  \inf_{\mathbf
u^h_{\mathrm S} \in \mathbf H^h_{\mathrm S}(\Omega_{\mathrm S})} \|
\mathbf u_{\mathrm S}-\mathbf u^h_{\mathrm S}\|_{\mathbf
H^1(\Omega_{\mathrm S})} + \inf_{\mathbf u^h_{\mathrm D}\in \mathbf
H^h_{\mathrm D}(\Omega_{\mathrm D})} \| \mathbf u_{\mathrm
D}-\mathbf u^h_{\mathrm D}\|_{\mathbf
H(\mathrm{div},\Omega_{\mathrm D})}\\
  + \,\lambda(h)\, \|\mathbf u_{\mathrm S}\cdot\boldsymbol\nu-R^h_D(\mathbf
u_{\mathrm S}\cdot\boldsymbol\nu)\|_{\Sigma},
\end{multline}
with $\lambda(h)\equiv 0$ if $\Phi^h_{\mathrm S}\subset \Phi^h_{\mathrm D}$ and 
$\lambda(h) \lesssim h^{1/2}$ otherwise.
\end{proposition}

\begin{proof}
Let
\[
\boldsymbol\Pi^h_{\mathrm S}:\mathbf H^1_{\mathrm S}(\Omega_{\mathrm
S}) \to \mathbf H^h_{\mathrm S}(\Omega_{\mathrm S}) \quad \text{and} \quad 
\boldsymbol\Pi^h_{\mathrm D}:\mathbf H_{\mathrm
D}(\mathrm{div},\Omega_{\mathrm D}) \to \mathbf H^h_{\mathrm
D}(\Omega_{\mathrm D})
\]
be the orthogonal projections onto the two discrete spaces
\eqref{eq:3.4}. Given $\mathbf u\equiv (\mathbf u_{\mathrm
S},\mathbf u_{\mathrm D})\in \mathbb X$, we consider 
\[
\mathbf u_h\equiv (\mathbf u^h_{\mathrm S},\mathbf u^h_{\mathrm
D}):=(\boldsymbol\Pi^h_{\mathrm S} \mathbf u_{\mathrm S},
\boldsymbol\Pi^h_{\mathrm D}\mathbf u_{\mathrm D}-\mathbf
L^h(\boldsymbol\Pi^h_{\mathrm D}\mathbf u_{\mathrm
D}\cdot\boldsymbol\nu-R^h_D(\boldsymbol\Pi^h_{\mathrm S}\mathbf
u_{\mathrm S}\cdot\boldsymbol\nu))).
\]
Since $\mathbf u^h_{\mathrm D}\cdot\boldsymbol\nu=
R^h_D(\boldsymbol\Pi^h_{\mathrm S}\mathbf u_{\mathrm
S}\cdot\boldsymbol\nu)=R^h_D(\mathbf u^h_{\mathrm
S}\cdot\boldsymbol\nu)$, it follows that $\mathbf u_h \in\mathbb
X^h$. On the other hand, the uniform boundedness of $\mathbf L^h$ yields 
\begin{multline}\label{eq:mia0}
\|\mathbf u-\mathbf u_h\|_{\mathbb X}\lesssim \|\mathbf u_{\mathrm
S}-\boldsymbol\Pi^h_{\mathrm S}\mathbf u_{\mathrm S}\|_{\mathbf
H^1(\Omega_{\mathrm S})}+\|\mathbf u_{\mathrm
D}-\boldsymbol\Pi^h_{\mathrm D}\mathbf u_{\mathrm D}\|_{\mathbf
H(\mathrm{div},\Omega_{\mathrm D})}+\\
\|\boldsymbol\Pi^h_{\mathrm D}
\mathbf u_{\mathrm D}\cdot\boldsymbol\nu- R^h_D(\boldsymbol\Pi^h_{\mathrm
S}\mathbf u_{\mathrm S}\cdot\boldsymbol\nu)\|_{H_{00}^{-1/2}(\Sigma)}
\end{multline}
and the triangle inequality together with the fact that $\mathbf u_{\mathrm
S}\cdot\boldsymbol\nu=\mathbf u_{\mathrm D}\cdot\boldsymbol\nu$ gives
\begin{multline}\label{eq:mia1}
 \|\boldsymbol\Pi^h_{\mathrm D}
\mathbf u_{\mathrm D}\cdot\boldsymbol\nu- R^h_D(\boldsymbol\Pi^h_{\mathrm
S}\mathbf u_{\mathrm S}\cdot\boldsymbol\nu)\|_{H_{00}^{-1/2}(\Sigma)} \leq
\|(\mathbf u_{\mathrm D} - \boldsymbol\Pi^h_{\mathrm D}
\mathbf u_{\mathrm D})\cdot\boldsymbol\nu \|_{H_{00}^{-1/2}(\Sigma)} + \\
\|\mathbf u_{\mathrm
S}\cdot\boldsymbol\nu - R^h_D(\boldsymbol\Pi^h_{\mathrm
S}\mathbf u_{\mathrm S}\cdot\boldsymbol\nu)\|_{H_{00}^{-1/2}(\Sigma)}. 
\end{multline}
Consequently, if $\Phi^h_{\mathrm S}\subset \Phi^h_{\mathrm D}$, $R^h_D$ 
is redundant in the last estimate. Consequently,     
\eqref{eq:mia} is satisfied in this case with $\lambda(h) = 0$ 
thanks to the boundedness of the normal trace operator. 

In the general case, we will need the estimate 
\begin{equation}\label{eq:mia2}
 \|\xi - R^h_D \xi\|_{H_{00}^{-1/2}(\Sigma)} \lesssim h^{1/2} \| \xi \|_{\Sigma}\quad \forall \xi \in L^2(\Sigma),
\end{equation}
obtained from  a duality argument by taking into account \eqref{eq:6.1}. 
We estimate the second term of the right hand side of \eqref{eq:mia1} by the triangle inequality 
\begin{multline*}
 \|\mathbf u_{\mathrm
S}\cdot\boldsymbol\nu - R^h_D(\boldsymbol\Pi^h_{\mathrm
S}\mathbf u_{\mathrm S}\cdot\boldsymbol\nu)\|_{H_{00}^{-1/2}(\Sigma)} \leq 
\|\mathbf u_{\mathrm
S}\cdot\boldsymbol\nu - R^h_D(\mathbf u_{\mathrm S}\cdot\boldsymbol\nu)\|_{H_{00}^{-1/2}(\Sigma)} + \\
\|\mathbf u_{\mathrm
S}\cdot\boldsymbol\nu -  \boldsymbol\Pi^h_{\mathrm
S}\mathbf u_{\mathrm S}\cdot\boldsymbol\nu \|_{H_{00}^{-1/2}(\Sigma)} + \|(\text{id} - R^h_D) 
(\mathbf u_{\mathrm
S}\cdot\boldsymbol\nu -  \boldsymbol\Pi^h_{\mathrm
S}\mathbf u_{\mathrm S}\cdot\boldsymbol\nu)\|_{H_{00}^{-1/2}(\Sigma)}
\end{multline*}
and by using  \eqref{eq:mia2} twice 
\[
 \|\mathbf u_{\mathrm
S}\cdot\boldsymbol\nu - R^h_D(\boldsymbol\Pi^h_{\mathrm
S}\mathbf u_{\mathrm S}\cdot\boldsymbol\nu)\|_{H_{00}^{-1/2}(\Sigma)}\lesssim h^{1/2} \, \|\mathbf u_{\mathrm
S}\cdot\boldsymbol\nu - R^h_D(\mathbf u_{\mathrm S}\cdot\boldsymbol\nu)\|_\Sigma + 
\| 
\mathbf u_{\mathrm
S} -  \boldsymbol\Pi^h_{\mathrm
S}\mathbf u_{\mathrm S}\|_\Sigma.
\]
The result follows by applying the trace theorem in $\mathbf{H}^1(\Omega_{\mathrm{S}})$ and combining the resulting estimate 
with \eqref{eq:mia0}.
\end{proof}

We are now in a position to establish our main result. 
\begin{theorem}\label{thm:main}
Hypotheses \ref{hyp:1} and \ref{hyp:2} imply unique solvability of
the discrete equations \eqref{eq:2.0} and the error estimate 
\begin{multline*}
\|\mathbf u-\mathbf u_h\|_{\mathbb
X}+\|p-p_h\|_{\Omega_{\mathrm
D}}+\|p+\delta-(p_h+\delta_h)\|_{\Omega_{\mathrm S}}
\lesssim  \\  \inf_{\mathbf
u^h_{\mathrm S} \in \mathbf H^h_{\mathrm S}(\Omega_{\mathrm S})} \|
\mathbf u_{\mathrm S}-\mathbf u^h_{\mathrm S}\|_{\mathbf
H^1(\Omega_{\mathrm S})} + \inf_{\mathbf u^h_{\mathrm D}\in \mathbf
H^h_{\mathrm D}(\Omega_{\mathrm D})} \| \mathbf u_{\mathrm
D}-\mathbf u^h_{\mathrm D}\|_{\mathbf
H(\mathrm{div},\Omega_{\mathrm D})}+ 
\inf_{q_h \in
L^2_\star(\Omega)}\| p-q_h\|_\Omega \\ + \lambda(h) \, \left( \|p_{\mathrm{D}} - R^h_D p_{\mathrm{D}} \|_\Sigma + 
\|\mathbf u_{\mathrm S}\cdot\boldsymbol\nu-R^h_D(\mathbf
u_{\mathrm S}\cdot\boldsymbol\nu)\|_{\Sigma}
\right),
\end{multline*}
where $\lambda(h)\equiv 0$ if $\Phi_{\mathrm S}^h \subset \Phi_{\mathrm D}^h$ and 
$
\mathrm \lambda(h) \lesssim h^{1/2} 
$
otherwise.
\end{theorem}

\section{The lifting of the normal trace}\label{sec:5}
Theorem \ref{thm:main} shows that, given a set of two mixed finite elements 
that are separately convergent for the Darcy and the Stokes problems, 
the only requirement for the convergence of scheme \eqref{eq:2.0} is  
Hypothesis \ref{hyp:2}. In this section, we discuss the conditions under which Hypothesis \ref{hyp:2}(b) 
is satisfied for the most common finite element subspaces of $\mathbf{H}(\mathrm{div}, \Omega_{\mathrm{D}})$. 

We begin by proving the existence of a stable lifting of the normal trace for Raviart-Thomas and 
Brezzi-Douglas-Marini spaces in two dimensions. The only restriction on the grid is shape-regularity. 
For the sake of clarity, we define all needed elements here. The geometric elements are a polygonal domain 
$\Omega\subset \mathbb R^2$ with boundary $\Gamma$, a shape-regular family of triangulations $\mathcal T_h$ , 
and the partitions $\Gamma_h$ of $\Gamma$ that are inherited from $\mathcal T_h$. The space of vector discrete 
vector fields is the following:
\[
\mathbf H^h:=\{ \mathbf v_h\in \mathbf H(\mathrm{div},\Omega)\,:\,
\mathbf v|_T \in \mathbf P(T)\quad\forall T\in \mathcal T_h\},
\]
where
\[
\mathbf P(T):=\left\{\begin{array}{ll} \mathbb P_0(T)^2+\mathrm{span}\,\{ \mathbf x\}, & \mbox{if $k=0$},\\
\mathbb P_k(T)^2, &\mbox{if $k\ge 1$}.\end{array}\right.
\]
Note that $\mathbf H^h$ is the velocity space for the  Raviart-Thomas pair when $k=0$ and for the BDM pair 
for $k\ge 1$. The latter is a strict subspace of the RT space of the same order.
The space of normal traces is
\[
\Phi^h:=\{ \xi_h:\Gamma \to \mathbb R\,:\, \xi_h|_e \in \mathbb P_k(e)\quad\forall e\in \Gamma_h\}.
\]
It is clear that $\mathbf u_h\cdot\boldsymbol\nu\in \Phi_h$ for all $\mathbf u_h \in \mathbf H^h$.

\begin{theorem} There exists $\mathbf L^h:\Phi^h\to \mathbf H^h$ such that
\begin{equation}\label{eq:8.1}
(\mathbf L^h\xi_h)\cdot\boldsymbol\nu=\xi_h \quad \text{and} \quad  \|\mathbf L^h\xi_h\|_{\mathbf H(\mathrm{div},\Omega)}
\lesssim \|\xi_h\|_{H^{-1/2}(\Gamma)} \qquad \forall \xi_h \in \Phi^h.
\end{equation}
In addition,
\begin{equation}\label{eq:8.2}
\mathrm{div} \mathbf L^h \xi_h = \frac1{|\Omega|} \int_\Gamma \xi_h.
\end{equation}
\end{theorem}

\begin{proof}
We start by decomposing
\begin{equation}\label{eq:8.3}
\xi_h = \frac1{|\Gamma|} \int_\Gamma\xi_h + \Big(\xi_h - \frac1{|\Gamma|} \int_\Gamma\xi_h\Big)
=c_h + \xi_h^0\in \mathbb P_0(\Gamma)\oplus \Big( \Phi_h \cap H^{-1/2}_0(\Gamma)\Big),
\end{equation}
where $H^{-1/2}_0(\Gamma):=\{\xi\in H^{-1/2}(\Gamma)\,:\, \langle \xi, 1\rangle_\Gamma =0\}$.

{\em Lifting of a constant function.} Consider the solution of the Neumann problem
\[
\Delta u \equiv |\Omega|^{-1} \quad \mbox{in $\Omega$}, \qquad \partial_{\boldsymbol\nu} u\equiv 
|\Gamma|^{-1}, \qquad \int_\Omega u =0
\]
and let $\mathbf w:=\nabla u$. From well known regularity results \cite{Gri92} it 
follows that $\mathbf w\in \mathbf H^{1/2+\varepsilon}(\Omega)$ for some $\varepsilon=
\varepsilon(\Omega)>0$. Let then $\mathbf w_h$ be the lowest order Raviart-Thomas projection of $\mathbf w$, i.e.,
\[
\mathbf w_h \in \mathbf H(\mathrm{div},\Omega), \quad \mathbf w_h \in 
\mathbb P_0(T)^2\oplus\mathrm{span}\{\mathbf x\} 
\quad \forall T \in \mathcal T_h, \quad \int_e \mathbf w_h \cdot\boldsymbol\nu_e =
\int_e \mathbf w \cdot\boldsymbol\nu_e\quad \forall e\in \mathcal E_h,
\]
where $\mathcal E_h$ is the set of edges of the triangulation and $\boldsymbol\nu_e$ is the unit normal vector on $e$ 
(for any given orientation). This projection is well defined because of the regularity of $\mathbf w$. It is then clear 
(using a direct argument or well known properties of the Raviart-Thomas element) that
\begin{equation}\label{eq:8.4}
\mathbf w_h\cdot\boldsymbol\nu\equiv |\Gamma|^{-1},\qquad \mathrm{div}\,\mathbf w_h \equiv |\Omega|^{-1}.
\end{equation}
A more delicate argument (see for instance \cite{bernardi}) shows that
\begin{equation}\label{eq:8.5}
\| \mathbf w_h\|_{\mathbf H(\mathrm{div},\Omega)}\lesssim \| \mathbf w\|_{\mathbf H(\mathrm{div},\Omega)}+
\|\mathbf w\|_{\mathbf{H}^{1/2+\varepsilon}(\Omega)}=:C_\Omega.
\end{equation}
{\em Lifting by arc-length integration.} The gist of the lifting process follows from the application of a projection on a 
space of continuous finite elements after integration of $\xi_h^0$ along the length of $\Gamma$. Consider the spaces 
\[
V^h:=\{ u_h \in \mathcal C(\overline\Omega)\,:\, u_h|_T \in \mathbb P_{k+1}(T) \quad \forall T \in \mathcal T_h\}, 
\qquad \Psi^h:=\gamma V^h=\{ \gamma u_h\,:\,u_h \in V^h\}
\]
where $\gamma$ stands here for the trace operator on $\Gamma$. 
In \cite{ScoZha90} an operator $S_h: H^1(\Omega) \to V^h$ is constructed with the following properties:
\begin{equation}\label{eq:8.6}
S_h^2 u=S_h u, \qquad \| S_h u\|_{H^1(\Omega)} \lesssim \| u\|_{H^1(\Omega)} \quad \forall u\in H^1(\Omega),
\end{equation}
and
\begin{equation}\label{eq:8.7}
\gamma u = 0 \qquad \Longrightarrow\qquad \gamma S_h u=0.
\end{equation}
Note that \eqref{eq:8.6} and \eqref{eq:8.7} imply that if
$\gamma u \in \Psi^h$, then $\gamma S_h u= \gamma u.$ (This property can be verified directly 
for the particular construction of 
Scott and Zhang \cite{ScoZha90}, or proved directly from properties \eqref{eq:8.6} and \eqref{eq:8.7}.)
Let then $D^{-1}:H^{-1/2}_0(\Gamma)\to H^{1/2}(\Gamma)$ be an inverse of the arc-length differentiation 
operator and let 
$L:H^{1/2}(\Gamma) \to H^1(\Omega)$ be a bounded right-inverse of the trace operator. 
Given $\xi_h^0\in \Phi^h \cap H^{-1/2}_0(\Gamma)$, 
we can define
\[
\mathbf v_h:=\nabla^\top S_h L D^{-1} \xi_h^0, \qquad \nabla^\top:=(\partial_y,-\partial_x).
\]
Note that $\gamma L D^{-1}\xi_h^0=D^{-1}\xi_h^0 \in \Psi^h$ and therefore $\gamma S_h  L D^{-1}\xi_h^0=D^{-1}\xi_h^0$. Hence
\begin{equation}\label{eq:8.8}
\mathbf v_h \cdot\boldsymbol\nu=( \gamma S_h L D^{-1} \xi_h^0)'=(D^{-1} \xi_h^0)'=\xi_h^0.
\end{equation}
It is also clear that $\mathbf v_h$ is piecewise $\mathbb P_k$ and that its divergence vanishes. 
Therefore $\mathbf v_h \in \mathbf H^h$. 
Finally,
\begin{multline}\label{eq:8.9}
\| \mathbf v_h\|_{\mathbf H(\mathrm{div},\Omega)}=\|\mathbf v_h\|_{\Omega}\le \| S_h L D^{-1} \xi_h^0\|_{H^1(\Omega)}\lesssim 
\| L D^{-1} \xi_h^0\|_{H^1(\Omega)}\\
\lesssim \| D^{-1}\xi_h^0\|_{H^{1/2}(\Gamma)} \lesssim \|\xi_h^0\|_{H^{-1/2}(\Gamma)}.
\end{multline}
{\em Conclusion.} The full lifting operator is then given by the expression
\[
\mathbf L^h \xi_h := \Big(\int_\Gamma \xi_h \Big)\mathbf w_h + \nabla^\top S_h L D^{-1} 
\Big( \xi_h - \frac1{|\Gamma|}\int_\Gamma \xi_h\Big),
\]
using the discrete lifting of a constant given above \eqref{eq:8.4}-\eqref{eq:8.5}.
By \eqref{eq:8.4} and \eqref{eq:8.8}, it follows that
\[
(\mathbf L^h\xi_h)\cdot\boldsymbol\nu= \frac1{|\Gamma|}\int_\Gamma 	\xi_h +  
\Big( \xi_h - \frac1{|\Gamma|}\int_\Gamma \xi_h\Big)=\xi_h.
\]
Also, by \eqref{eq:8.5} and \eqref{eq:8.9}, it follows that
\[
\| \mathbf L^h\xi_h\|_{\mathbf H(\mathrm{div},\Omega)} \lesssim \left| \int_\Gamma \xi_h\right| + 
\|\xi_h\|_{H^{-1/2}(\Gamma)}\lesssim  \|\xi_h\|_{H^{-1/2}(\Gamma)}.
\]
Finally, by \eqref{eq:8.4}
\[
\mathrm{div}\mathbf L^h \xi_h= \Big(\int_\Gamma \xi_h\Big)
\mathrm{div}\,\mathbf w_h\equiv  \frac1{|\Omega|} \int_\Gamma \xi_h,
\]
which finishes the proof.
\end{proof}

\begin{remark}
In the three dimensional case a stable lifting of the normal trace can be constructed 
for RT and BDM elements using an additional 
hypothesis on $\mathcal T_h$, namely, that $\mathcal T_h$ is quasiuniform on a neighborhood of 
$\Gamma$. A proof for RT of any order 
in two dimensions is given in \cite{GatOyaSayTA}: the proof holds for three dimensions and for 
BDM elements.
\end{remark}

\section{The discrete permeability condition}\label{sec:6}
In this section we discuss if 
Hypothesis \ref{hyp:2}(a) is available in practical situations. This condition 
can be understood as the possibility of having flow across $\Sigma$ without the need of an
increasingly large velocity field, that is, the discrete boundary
created by the choice of spaces has to be permeable enough. Let us first 
show that Hypothesis \ref{hyp:2}(a) can be easily deduced from Hypothesis \ref{hyp:2}(b).
\begin{proposition}\label{prop:mia}
 Assume that Hypothesis \ref{hyp:2}(b) is satisfied and that 
\[
\omega(h) := \inf_{\mathbf
u^h_{\mathrm S} \in \mathbf H^h_{\mathrm S}(\Omega_{\mathrm S})} \|
\mathbf u_{\mathrm S}-\mathbf u^h_{\mathrm S}\|_{\mathbf
H^1(\Omega_{\mathrm S})} + \inf_{\mathbf u^h_{\mathrm D}\in \mathbf
H^h_{\mathrm D}(\Omega_{\mathrm D})} \| \mathbf u_{\mathrm
D}-\mathbf u^h_{\mathrm D}\|_{\mathbf
H(\mathrm{div},\Omega_{\mathrm D})}                                                               
\]
converges to zero when the parameter $h$ tends to zero. Then, there exists $h_0>0$ such that 
Hypothesis \ref{hyp:2}(a) is satisfied for all $h\leq h_0$
\end{proposition}
\begin{proof}
 Let $\mathbf{u}\in \mathbf{H}^1_0(\Omega)$ be such that 
\[
 \|\mathbf{u}\|_{\mathbb{X}} \leq 1 \qquad \text{and}\qquad 
\langle\mathbf{u}\cdot \boldsymbol{\nu}, 1\rangle_\Sigma \geq 1.
\]
With the notations of Section \ref{sec:4} we define  
\[
\mathbf u_h\equiv (\mathbf u^h_{\mathrm S},\mathbf u^h_{\mathrm
D}):=(\boldsymbol\Pi^h_{\mathrm S} \mathbf u_{\mathrm S},
\boldsymbol\Pi^h_{\mathrm D}\mathbf u_{\mathrm D}-\mathbf
L^h(\boldsymbol\Pi^h_{\mathrm D}\mathbf u_{\mathrm
D}\cdot\boldsymbol\nu-R^h_D(\boldsymbol\Pi^h_{\mathrm S}\mathbf
u_{\mathrm S}\cdot\boldsymbol\nu))),
\]
where $\mathbf u_{\mathrm D} := \mathbf{u}|_{\Omega_{\mathrm D}}$ and 
$\mathbf u_{\mathrm S} := \mathbf{u}|_{\Omega_{\mathrm S}}$. We know from Proposition \ref{prop:5.4} that 
\[
\| \mathbf
u-\mathbf u_h\|_{\mathbb X}
 \lesssim \omega(h) + \,\lambda(h)\, \|\mathbf u_{\mathrm S}\cdot\boldsymbol\nu-R^h_D(\mathbf
u_{\mathrm S}\cdot\boldsymbol\nu)\|_{\Sigma},
\]
with $\lambda(h) \lesssim h^{1/2}$. It follows that 
\[
 \| \mathbf{u}_h \|_{\mathbb{X}} \leq \| \mathbf{u}_h \|_{\mathbb{X}} + \| \mathbf u-\mathbf u_h\|_{\mathbb X} \leq 2 
\]
and 
\[
\langle\mathbf{u}_h\cdot \boldsymbol{\nu}, 1\rangle_\Sigma
=  
\langle\mathbf{u}\cdot \boldsymbol{\nu}, 1\rangle_\Sigma -
\langle(\mathbf{u} - \mathbf{u}_h)\cdot \boldsymbol{\nu}, 1\rangle_\Sigma 
\gtrsim 1 - \| \mathbf u-\mathbf u_h\|_{\mathbb X} \gtrsim 1/2
\]
if $h$ is sufficiently small.
\end{proof}

Next, we will see that Hypothesis \ref{hyp:2}(a) is actually quite mild, it can be shown to
hold on (independently from Hypothesis \ref{hyp:2}(b)) with few restrictions.  
In some of the forthcoming arguments and examples we will use the space of continuous finite elements:
\[
\mathbf P^{\mathrm{cont}}_{k}(\mathcal T^h_{\mathrm S}):=\{ \mathbf v \in \mathcal C(\overline{\Omega}_{\mathrm S})^d\,:\, 
\mathbf v|_T \in \mathbb P_k(T) \quad \forall T \in \mathcal T^h_{\mathrm S}\}.
\]

\begin{proposition}\label{prop:3.1}
In the two dimensional case, assume that:
\begin{itemize}
\item[{\rm (a)}] either $\Sigma$ is a polygon with two or more edges,
\item[{\rm (b)}] or $\Sigma$ is a line segment and there exists a fixed 
node $\mathbf p$ belonging to all partitions $\Sigma^h_{\mathrm S}$ such that the distance of $\mathbf p$ to
$\partial\Omega$ remains bounded below.
\end{itemize}
Additionally, assume that
\begin{equation}\label{eq:3.1}
\mathbf P^{\mathrm{cont}}_{1}(\mathcal T^h_{\mathrm S}) \subset \mathbf
H^h(\Omega_{\mathrm S}).
\end{equation}
Then Hypothesis \ref{hyp:2}(a) is satisfied.
\end{proposition}

\begin{proof}
We start with the geometric condition (a). We choose two adjacent
edges, $\Sigma_1$ and $\Sigma_2$, of $\Sigma$ and let $\mathbf p:=
\Sigma_1\cap\Sigma_2$. We next construct the function
$\boldsymbol\rho:\Sigma \to \mathbb R^2$ given by
\[
\boldsymbol\rho:= \smallfrac12 (\boldsymbol\nu_1+\boldsymbol\nu_2)
\psi,
\]
where $\boldsymbol\nu_i$ is the normal vector on $\Sigma_i$ and
$\psi:\Sigma \to \mathbb R$ is a continuous piecewise linear
function (piecewise with respect to the natural partition of
$\Sigma$ in edges and linear with respect to the arc
parameterization) such that $\psi(\mathbf p)=1$ and $\psi$ vanishes
on all other vertices of $\Sigma$. Note first that
\begin{equation}\label{eq:3.2}
\int_\Sigma \boldsymbol\rho\cdot\boldsymbol\nu = \frac12
\int_{\Sigma_1\cup\Sigma_2}
(1+\boldsymbol\nu_1\cdot\boldsymbol\nu_2) \, \psi = \frac14
(1+\boldsymbol\nu_1\cdot\boldsymbol\nu_2)\,
|\Sigma_1\cup\Sigma_2|=:c_0>0
\end{equation}
with $c_0$ independent of any discrete quantity. We next use the
interpolation operator for non-smooth functions of Scott and Zhang
\cite{ScoZha90} (see also \cite{BreSco08} and \cite{DomSay03}) to construct $\mathbf v_h
\in \mathbf P_{1}^{\mathrm{cont}}(\mathcal T^h_{\mathrm S})$ such that
$\mathbf v_h=\boldsymbol\rho$ on $\Sigma$ and $\mathbf v_h=0$ on
$\Gamma_{\mathrm S}$. Because the Scott-Zhang lifting operator is
stable, it follows that
\begin{equation}\label{eq:3.3}
\|\mathbf v_h\|_{\mathbf{H}^1(\Omega_{\mathrm S})}\lesssim
\|\boldsymbol\rho\|_{\mathbf H_{00}^{1/2}(\Sigma)}=:C_0,
\end{equation}
where $C_0$ independent of any discrete quantities. Since
\eqref{eq:3.1} is satisfied, this proves the hypothesis.

In case (b), we can choose $\Sigma_1$ and $\Sigma_2$ to be the
segments joining $\mathbf p$ to the boundary of $\Omega$ and proceed
similarly. In this case, the quantity $C_0$ in \eqref{eq:3.3}
depends on $\mathbf p$: the hypothesis on the distance of $\mathbf
p$ to $\partial\Omega$ is equivalent to having
$\min\{|\Sigma_1|,|\Sigma_2|\}\ge c_1>0$, which can be used to give
an upper bound of $\|\boldsymbol\rho\|_{\mathbf H_{00}^{1/2}(\Sigma)}$.
\end{proof}

\begin{remark} The idea of Proposition \ref{prop:3.1} can be extended to cover more cases:
\begin{itemize}
\item[(a)] If $\mathbb P_{d}^{\mathrm{cont}}(\mathcal T^h_{\mathrm S})\subset \mathbf H^h(\Omega_{\mathrm S})$  
(as usual, $d$ is the dimension of the physical space), a simplified proof can be carried out by choosing a face 
of $\Sigma$, building a bubble function $\boldsymbol\rho$ in the normal direction and extending it to $\Omega_{\mathrm S}$  
using the Scott-Zhang interpolation method. If $\Sigma$ does not contain any triangular face, we 
have to additionally assume there exists a fixed triangle, that can always be obtained as the union of elements (triangles) 
from $\Sigma^h_{\mathrm S}$. 
\item[(b)] In the three dimensional case, with lower order polynomials, it is easy to find geometric configurations 
of $\Sigma$ that match the requirements of Proposition \ref{prop:3.1}, so that we can construct a discrete bubble 
function on an averaged normal direction on part of the interface.
\end{itemize}
\end{remark}

\section{Examples}\label{sec:7}

In this section we provide several examples of pairs of stable elements for the Stokes-Darcy problem. 
For careful description of the associated mixed elements on the Darcy domain and for stable finite elements for the Stokes 
problem, the reader is referred to \cite{BreFor91} and \cite{ErnGue04}. Original sources for these elements 
(many of them discovered in several steps and renamed several times) can be found in these already classical references. 
Spaces in the Stokes and Darcy domain will be chosen so that  Hypothesis \ref{hyp:1} is satisfied.

\subsection{The conforming case}

We next list several examples of choices where the hypotheses above are met. In all the examples below, the space for the 
Darcy domain will be the Brezzi-Douglas-Marini (BDM) space, sometimes referred to as the Brezzi-Douglas-Dur\'an-Fortin in 
the three dimensional case. For a triangulation $\mathcal T^h_{\mathrm D}$ of $\Omega_{\mathrm D}$, we consider the spaces 
for $k\ge 1$:
\begin{eqnarray*}
\mathbf H^h(\Omega_{\mathrm D})&:=&\{\mathbf u_h \in \mathbf H(\mathrm{div},\Omega_{\mathrm D})\,:\,\mathbf u_h|_T\in 
\mathbb P_{k}(T)^d \quad\forall T\in \mathcal T^h_{\mathrm D}\},\\
 L^h(\Omega_{\mathrm D})&:=&\{ p_h :\Omega_{\mathrm D}\to \mathbb R\,:\, p_h|_T \in \mathbb P_{k-1}(T) \quad\forall 
T\in \mathcal T^h_{\mathrm D}\}.
\end{eqnarray*}
We will refer to this choice with the generic name BDM($k$).
If the inherited triangulation of $\Sigma$ is denoted $\Sigma^h_{\mathrm D}$, the space $\Phi^h_{\mathrm D}$ consists 
of piecewise $\mathbb P_k$ functions. For conditions on when there is a stable discrete lifting, the reader is referred 
to Section \ref{sec:5}.

In the Stokes domain, we consider another triangulation $\mathcal T^h_{\mathrm S}$, producing a partition 
$\Sigma^h_{\mathrm S}$ of the interface. {\sl We assume that the Darcy partition $\Sigma^h_{\mathrm D}$ is either 
equal to or a refinement of $\Sigma^h_{\mathrm S}$.}  Since all discrete spaces for the velocity variable in the 
Stokes domain contain polynomial linear functions, we are going to assume that the geometry allows for Hypothesis 
\ref{hyp:2}(a) to be satisfied. 

A list of possible choices is given in Table \ref{table:1}. 

\begin{table}[h]\label{table:1}
\begin{center}
\begin{tabular}{|c|c|c||c|c|c||c|}
\hline
$\phantom{\Big|}\hspace{-0.2cm}$ Stokes & Velocity & Press. & Darcy & Vel. & Press. & Order\\
\hline
$\phantom{\Big|}\hspace{-0.2cm}$ MINI & $\mathbb P_1$+bubbles & $\mathbb P_1^\mathrm{cont}$ & BDM(1) & $\mathbb P_1$ & $\mathbb P_0$ & $h$\\
$\phantom{\Big|}\hspace{-0.2cm}$ Taylor-Hood, $k\ge 2$ & $\mathbb P_k$ & $\mathbb P_{k-1}^\mathrm{cont}$ & BDM($k$) & $\mathbb P_k$ & $\mathbb P_{k-1}$ & $h^k$\\
$\phantom{\Big|}\hspace{-0.2cm}$ Conf Crouzeix-Raviart & $\mathbb P_2$+bubbles & $\mathbb P_{1}$ & BDM(2) & $\mathbb P_2$ & $\mathbb P_1$ & $h^2$\\
$\phantom{\Big|}\hspace{-0.2cm}$ Bernardi-Raugel & $\mathbb P_1$+face bubbles & $\mathbb P_0$ & BDM(1) & $\mathbb P_1$ & $\mathbb P_0$ & $h$\\
\hline
\end{tabular}
\end{center}\caption{Coupling of Stokes elements with BDM elements. The superscript ${}^\mathrm{cont}$ refers to the demand of 
continuity for the discrete pressure space. The bubbles are used for velocities in the MINI and conformal CR elements:  an internal 
$\mathbb P_{d+1}(T)$ bubble is added to the velocity space on each element. For the BR element, face bubbles are included on all internal 
faces, but no bubbles are added on faces lying on $\Sigma$. When these bubbles (no needed for stability) are added, the method stops being 
a particular case of this class.}
\end{table}

\subsection{The nonconforming case}

In addition to the BDM element, we now consider the Raviart-Thomas element (also called Raviart-Thomas-N\'ed\'elec in the three 
dimensional case). The spaces for the RT($k$) pair with $k\ge 0$ are
\begin{eqnarray*}
\mathbf H^h(\Omega_{\mathrm D})&:=&\{\mathbf u_h \in \mathbf H(\mathrm{div},\Omega_{\mathrm D})\,:\,\mathbf u_h|_T\in 
\underbrace{\mathbb P_{k}(T)^d\oplus \mathbb P_k(T)\{ \mathbf x\}}_{=\mathrm{RT}_k(T)} \quad\forall T\in \mathcal T^h_{\mathrm D}\},\\
 L^h(\Omega_{\mathrm D})&:=&\{ p_h :\Omega_{\mathrm D}\to \mathbb R\,:\, p_h|_T \in \mathbb P_{k}(T) \quad\forall T\in 
\mathcal T^h_{\mathrm D}\}.
\end{eqnarray*}
For this kind of coupling the Stokes and Darcy triangulations can be taken to be completely independent 
(although this might seriously complicate implementation in the three dimensional case). 
Condition \eqref{eq:6.1} is satisfied by the RT($k$) spaces $k\ge 0$ and the BDM($k$) spaces $k\ge 1$. 
For conditions concerning the lifting of the normal trace, see Section \ref{sec:5}. All Stokes spaces contain piecewise linear 
functions, which means that except in trivial geometric configurations Hypothesis \ref{hyp:2}(a) will be satisfied. 

\begin{table}[h]\label{table:2}
\begin{center}
\begin{tabular}{|c|c|c||c|c|c||c|}
\hline
$\phantom{\Big|}\hspace{-0.2cm}$ Stokes & Velocity & Press. & Darcy & Vel. & Press. & Order\\
\hline
$\phantom{\Big|}\hspace{-0.2cm}$ MINI & $\mathbb P_1$+bubbles & $\mathbb P_1^\mathrm{cont}$ & RT(0) & $\mathrm{RT}_0$ & $\mathbb P_0$ & $h$\\
$\phantom{\Big|}\hspace{-0.2cm}$ Taylor-Hood, $k\ge 2$ & $\mathbb P_k$ & $\mathbb P_{k-1}^\mathrm{cont}$ & RT($k-1$) & $\mathrm{RT}_{k-1}$ & $\mathbb P_{k-1}$ & $h^k$\\
$\phantom{\Big|}\hspace{-0.2cm}$ Bernardi-Raugel & $\mathbb P_1$+face bubbles & $\mathbb P_0$ & RT(0) & $\mathrm{RT}_0$ & $\mathbb P_0$ & $h$\\
$\phantom{\Big|}\hspace{-0.2cm}\mathbb P_2\mbox{-iso-}\mathbb P_1$ & $\mathbb P_1(\mathcal T^{h/2})$ & $\mathbb P_1^{\mathrm{cont}}$ & BDM(1) & $\mathbb P_1$ & $\mathbb P_0$ & $h$\\
\hline
\end{tabular}
\end{center}\caption{Coupling of Stokes elements with BDM and RT elements and their order of convergence. 
The superscript ${}^\mathrm{cont}$ refers to the demand of continuity for the discrete pressure space. 
The bubbles are used for velocities in the MINI  element. For the BR element, face bubbles are only included on the internal faces. 
Adding them to faces on $\Sigma$ does not change the convergence order. In that case BR can be coupled with BDM(1) as well.}
\end{table}

\section{Numerical results}\label{sec:8}

In order to confirm the good performance of our scheme (3.4), we present in this section the
combination of several stable Stokes elements with the Raviart-Thomas element and the
Brezzi-Douglas-Marini element, as shown in Tables \ref{table:1} and \ref{table:2}.
We begin by introducing some notations. 
The variable $N$ stands for the total number of degrees of freedom defining the 
finite element subspaces $\mathbb{X}^h$ and $\mathbb{Q}^h$, and 
the individual errors are denoted by:
\[
 \mathrm{e}(\mathbf{u}_{\mathrm D}) := \| \mathbf{u}_{\mathrm D} - \mathbf{u}^h_{\mathrm D} \|_{
\mathbf{H}(\mathrm{div}, \Omega_{\mathrm D})}, \qquad 
\mathrm{e}(\mathbf{u}_{\mathrm S}) := \| \mathbf{u}_{\mathrm S} - \mathbf{u}^h_{\mathrm S} \|_{
 \mathbf H^1(\Omega_{\mathrm S})
},
\]
and
 \[
  \mathrm{e}(p_{\mathrm D}) := \| p_{\mathrm D} - p^h_{\mathrm D}\|_{\Omega_{\mathrm D}},\qquad 
\mathrm{e}(p_{\mathrm S}) := \| p_{\mathrm S} - p^h_{\mathrm S}\|_{\Omega_{\mathrm S}},
 \]
where $\mathbf{u}^h_{\mathrm D} := \mathbf u_h|_{\Omega_{\mathrm D}}$, $\mathbf{u}^h_{\mathrm S} := \mathbf u_h|_{\Omega_{\mathrm D}}$, 
$p^h_{\mathrm D} := p_h|_{\Omega_{\mathrm D}}$ and $p^h_{\mathrm S} := p_h|_{\Omega_{\mathrm S}} + \delta_h$ with  
$(\mathbf u_h,(p_h,\delta_h))\in \mathbb X^h\times \mathbb Q^h$ being the solution of \eqref{eq:2.0}. We also let 
$r(\mathbf{u}_{\mathrm D})$, $r(\mathbf{u}_{\mathrm S})$, $r(p_{\mathrm D})$ and $r(p_{\mathrm S})$ be the experimental 
rates of convergence given by 
\[
 r(\mathbf{u}_{\mathrm D}) := \frac{\log(\mathrm{e}(\mathbf{u}_{\mathrm D})/ \mathrm{e}'(\mathbf{u}_{\mathrm D}) )}{\log(h/h')},\qquad
 r(\mathbf{u}_{\mathrm S}) := \frac{\log(\mathrm{e}(\mathbf{u}_{\mathrm S})/ \mathrm{e}'(\mathbf{u}_{\mathrm S}) )}{\log(h/h')},
\]
and
\[
 r(p_{\mathrm D}) := \frac{\log(\mathrm{e}(p_{\mathrm D})/ \mathrm{e}'(p_{\mathrm D}) )}{\log(h/h')},\qquad
 r(p_{\mathrm S}) := \frac{\log(\mathrm{e}(p_{\mathrm S})/ \mathrm{e}'(p_{\mathrm S}) )}{\log(h/h')},
\]
where $h$ and $h'$ are two consecutive mesh sizes with errors $\mathrm{e}$ and $\mathrm{e}'$.

\begin{figure}[ht!]
\footnotesize
\caption{MINI--BDM(1) coupling (left) and Bernardi-Raugel--RT(0) coupling (right).}
\label{figure:1}

\begin{center}
\begin{minipage}{7.5cm}
\centerline{\includegraphics[width=7cm, height=4cm, angle=00]{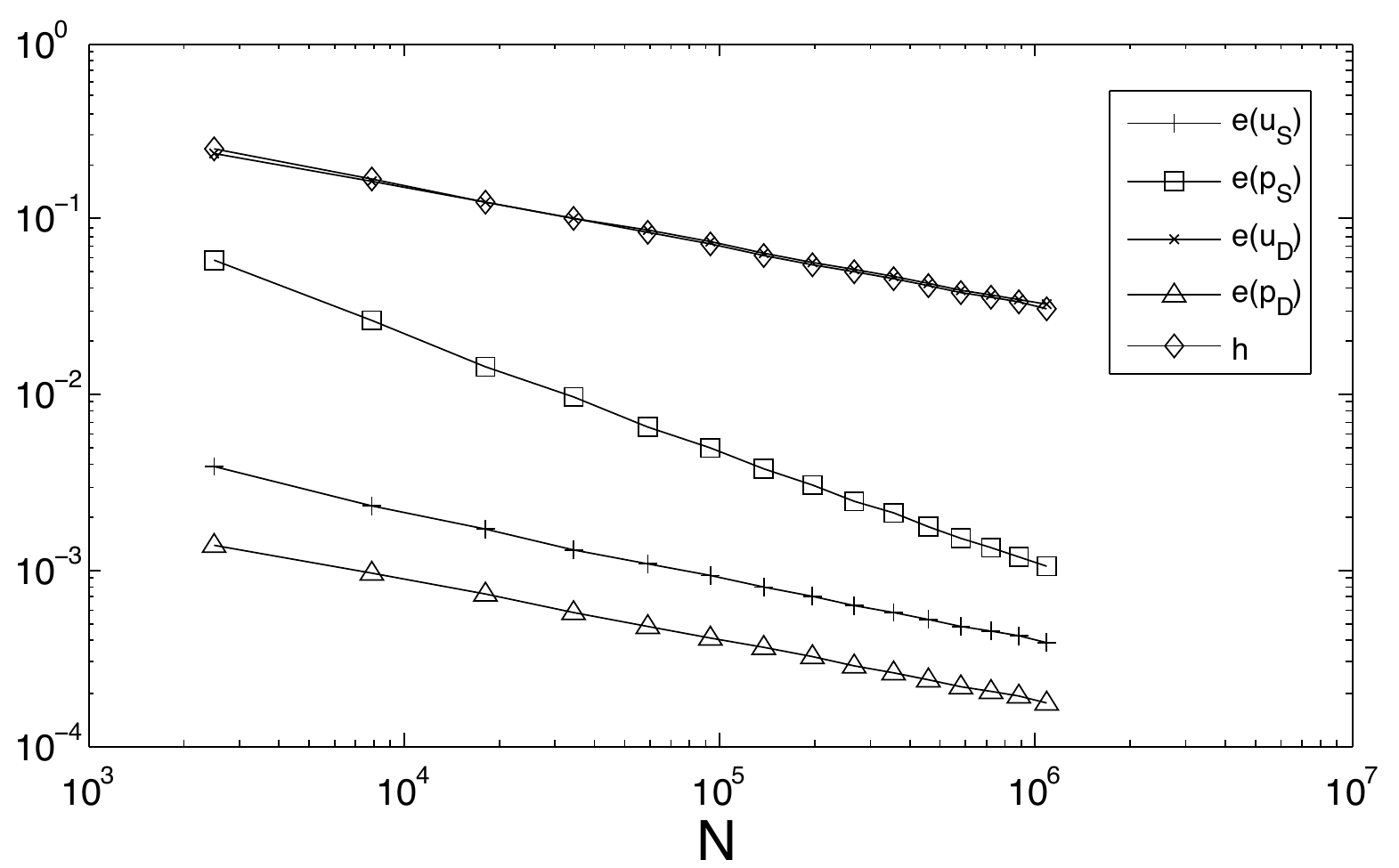}}
\end{minipage}
\begin{minipage}{7.5cm}
\centerline{\includegraphics[width=7cm, height=4cm, angle=00]{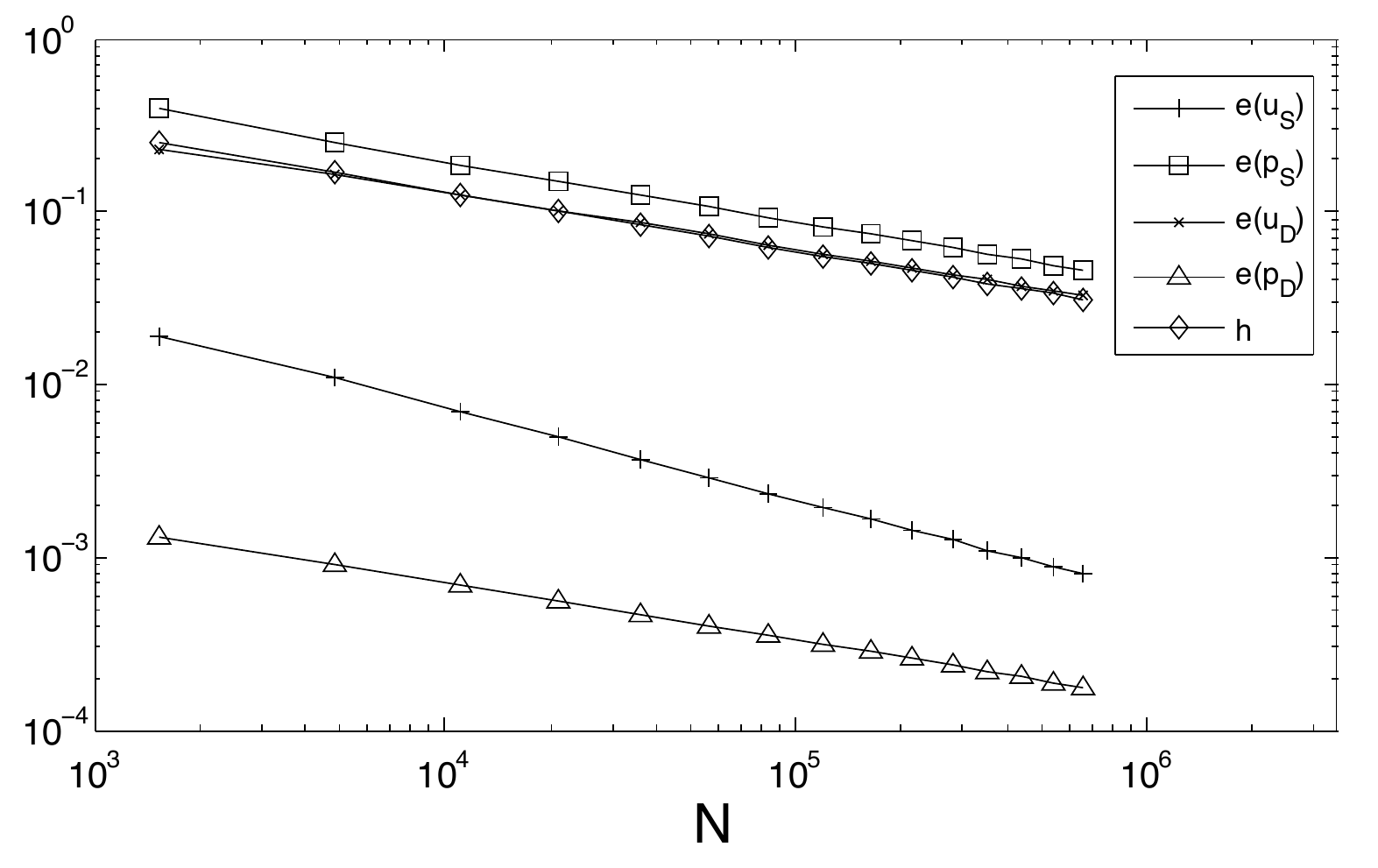}}
\end{minipage}
\end{center}
\end{figure}

\begin{table}[ht!]
\footnotesize
\caption{Convergence rates.}
\label{table:3}

\begin{center}
\begin{tabular}{|cc||cccc||cccc|}

\hline

& & \multicolumn{3}{c}{MINI--BDM(1)} & & \multicolumn{3}{c}{MINI--RT(0)} & \\ \hline

$h_\textrm{S}$  & $h_\textrm{D}$  & $r(\textbf{u}_\textrm{S})$
& $r(\textbf{u}_\textrm{D})$ & $r(p_\textrm{S})$ & $r(p_\textrm{D})$ 
& $r(\textbf{u}_\textrm{S})$
& $r(\textbf{u}_\textrm{D})$ & $r(p_\textrm{S})$ & $r(p_\textrm{D})$\\ \hline

$1/6$ & $1/12$ & $-$ & $-$ & $-$ & $-$ & $-$ & $-$ & $-$ & $-$ \\ \hline

$1/10$ &$1/20$ & 1.091 & 0.983 & 1.990 & 1.017 & 1.091 & 0.979 & 1.990 & 0.988 \\ \hline

$1/14$ &$1/28$ & 1.063 & 0.993 & 1.955 & 1.010 & 1.063 & 0.991 & 1.955 & 0.995 \\ \hline

$1/18$ &$1/36$ & 1.047 & 0.996 & 1.926 & 1.006 & 1.047 & 0.995 & 1.926 & 0.997 \\ \hline

\end{tabular}
\end{center}
\end{table}

\begin{table}[ht!]
\footnotesize
\caption{Convergence rates.}
\label{table:4}

\begin{center}
\begin{tabular}{|cc||cccc||cccc|}
\hline

& & \multicolumn{3}{c}{MINI--BDM(1)} &  & \multicolumn{3}{c}{MINI--RT(0)} & \\ \hline

$h_\textrm{S}$  & $h_\textrm{D}$  & $r(\textbf{u}_\textrm{S})$
&$r(\textbf{u}_\textrm{D})$&$r(p_\textrm{S})$ & $r(p_\textrm{D})$ & 
$r(\textbf{u}_\textrm{S})$
&$r(\textbf{u}_\textrm{D})$&$r(p_\textrm{S})$ & $r(p_\textrm{D})$\\ \hline

$1/12$ & $1/6$ & $-$ & $-$ & $-$ & $-$ & $-$ & $-$ & $-$ & $-$ \\ \hline

$1/20$ &$1/10$ & 1.063 & 0.933 & 1.907 & 0.996 & 1.061 & 0.921 & 1.907 & 0.951 \\ \hline

$1/28$ &$1/14$ & 1.037 & 0.971 & 1.856 & 1.021 & 1.037 & 0.966 & 1.856 & 0.980 \\ \hline

$1/36$ &$1/18$ & 1.027 & 0.984 & 1.824 & 1.018 & 1.026 & 0.981 & 1.824 & 0.989 \\ \hline

\end{tabular}
\end{center}
\end{table}

We now describe the data of the example. We consider the domains
$\Omega_\mathrm{D} := (0,1)^2 \times (0,0.5)^3$ and $\Omega_\mathrm{S} := (0,1)^2 \times
(0.5,1)$, and take $\nu =1$, $\kappa = 1$ and $\textbf{K} = \textbf{I}$, the identity of $\mathbb{R}^{3 \times 3}$. Non-homogeneous 
transmission conditions are considered in order to have an exact solution given by:
\[
p_\textrm{D}(\boldsymbol{x}):= x_1(1-x_1)\sin(2\pi x_1)\,x_2(1-x_2)\sin(2\pi x_2)\,x_3 \sin(2\pi x_3)-p_{\textrm{D}0}\,,
\]
in the porous media and by
$$
\textbf{u}_\textrm{S}(\boldsymbol{x}) := x_1(1-x_1)x_2(1-x_2)x_3(1-x_3)\left(\begin{array}{c}
                                                               -2x_1(1-x_1)(1-2x_2)(1-2x_3) \\
                                                               x_2(1-x_2)(1-2x_1)(1-2x_3) \\
                                                               x_3(1-x_3)(1-2x_1)(1-2x_2)
                                                             \end{array}
\right)\,,
$$
and
$$
p_\textrm{S}(\boldsymbol{x}) := \exp(x_1+x_2+x_3)\:,
$$
in the Stokes domain.

\begin{table}[ht!]
\footnotesize
\caption{Convergence rates.}
\label{table:5}

\begin{center}
\begin{tabular}{|cc||cccc||cccc|}
\hline

& & \multicolumn{3}{c}{Bernardi-Raugel--BDM(1)} & & \multicolumn{3}{c}{Bernardi-Raugel--RT(0)} & \\ \hline
$h_\textrm{S}$  & $h_\textrm{D}$  & 
$r(\textbf{u}_\textrm{S})$
&$r(\textbf{u}_\textrm{D})$&$r(p_\textrm{S})$ & $r(p_\textrm{D})$ & 
$r(\textbf{u}_\textrm{S})$
&$r(\textbf{u}_\textrm{D})$&$r(p_\textrm{S})$ & $r(p_\textrm{D})$\\ \hline

$1/6$ & $1/12$ & $-$ & $-$ & $-$ & $-$ & $-$ & $-$ & $-$ & $-$ \\ \hline

$1/10$ &$1/20$ & 1.537 & 0.983 & 1.034 & 1.020 & 1.537 & 0.979 & 1.038 & 0.991 \\ \hline

$1/14$ &$1/28$ & 1.574 & 0.993 & 1.021 & 1.011 & 1.574 & 0.991 & 1.021 & 0.996 \\ \hline

$1/18$ &$1/36$ & 1.574 & 0.996 & 1.014 & 1.007 & 1.574 & 0.995 & 1.014 & 0.998 \\ \hline

\end{tabular}
\end{center}
\end{table}

\begin{table}[ht!]
\footnotesize
\caption{Convergence rates.}
\label{table:6}

\begin{center}
\begin{tabular}{|cc||cccc||cccc|}
\hline
& & \multicolumn{3}{c}{Bernardi-Raugel--BDM(1)} & & \multicolumn{3}{c}{Bernardi-Raugel--RT(0)} & \\ \hline
\hline $h_\textrm{S}$  & $h_\textrm{D}$  & 
$r(\textbf{u}_\textrm{S})$
&$r(\textbf{u}_\textrm{D})$&$r(p_\textrm{S})$ & $r(p_\textrm{D})$ & 
$r(\textbf{u}_\textrm{S})$
&$r(\textbf{u}_\textrm{D})$&$r(p_\textrm{S})$ & $r(p_\textrm{D})$\\ \hline

$1/12$ & $1/6$ & $-$ & $-$ & $-$ & $-$ & $-$ & $-$ & $-$ & $-$ \\ \hline

$1/20$ &$1/10$ & 1.573 & 0.933 & 1.014 & 0.996 &  1.573 & 0.921& 1.014 & 0.951 \\ \hline

$1/28$ &$1/14$ &  1.543 & 0.971 & 1.008 & 1.021 &  1.543 & 0.966 & 1.008 & 0.980 \\ \hline

$1/36$ &$1/18$ &  1.503 & 0.984 & 1.005 & 1.018 & 1.503 & 0.981 & 1.005 & 0.989 \\ \hline

\end{tabular}
\end{center}
\end{table}

The numerical results were obtained using a MATLAB code. In Figure \ref{figure:1} we summarize the convergence 
history of the Galerkin scheme (3.4) for a sequence of uniform meshes of the
computational domain $\Omega := (0,1)^3$ by means of tetrahedra. We select the conforming example consisting in  
the MINI--BDM(1) coupling and the nonconforming example given by the 
 Bernardi-Raugel--RT(0) coupling. In each case we display the individual errors  versus the degrees of
freedom $N$. We observe  that, as expected,  the convergence is linear with respect to the discretization 
parameter $h$  in all the unknowns unless for the fluid pressure  in the MINI--BDM(1) example   
where a quadratic convergence is attained. We notice that the Bernardi-Raugel--RT(0) case  
delivers a convergence in the fluid velocity that is slightly faster than $O(h)$.

We also provide numerical results for triangulations with hanging nodes on the transmission 
interface. We consider, uniform triangulations of the subdomains $\Omega_D := (0,1/2)^3$ 
and $\Omega_S := (1/2,1)^3$ with a mesh size in one of the subdomains equal to half the mesh 
size in the other one. The expected rates of convergence are attained in all the (conforming and 
non-conforming) cases considered through Tables \ref{table:3}, \ref{table:4}, \ref{table:5} and \ref{table:6}.  

Summarizing, the numerical results presented here constitute enough support to our 
theory for the strong mixed finite element coupling of Darcy--Stokes flow problem. 
In a forthcoming work we will discuss an efficient 
iterative method to solve the linear system of equations arising from our 
discretization method.

%\bibliographystyle{abbrv}
%\bibliography{RefsStokes}

\end{document}